\documentclass[sn-mathphys-num]{sn-jnl}

\usepackage{graphicx}%
\usepackage{multirow}%
\usepackage{amsmath,amssymb,amsfonts}%
\usepackage{amsthm}%
\usepackage{mathrsfs}%
\usepackage[title]{appendix}%
\usepackage{enumitem}
\usepackage{booktabs}
\usepackage{mdframed}
\usepackage{xcolor}%
\usepackage{doi}
\usepackage{textcomp}%
\usepackage{manyfoot}%
\usepackage{booktabs}%
\usepackage{siunitx} 
\usepackage{algorithm}%
\usepackage{algorithmicx}%
\usepackage{comment}
\usepackage{algpseudocode}%
\usepackage{multirow}
\usepackage{chngcntr}
\counterwithin{algorithm}{section} 
\usepackage{listings}%
\usepackage{tikz}
\usepackage{subcaption} 
\usepackage{adjustbox} 
\usepackage{pgfplots}
\usepackage{booktabs}
\usepackage{array} 
\usepackage{float}  

\newtheorem{theorem}{Theorem}[section] 
\numberwithin{theorem}{section}

\theoremstyle{thmstyleone}%
\newtheorem{corollary}[theorem]{Corollary}

\newtheorem{proposition}[theorem]{Proposition}%
\newtheorem{assumption}{Assumption}%

\theoremstyle{thmstyletwo}%
\newtheorem{example}{Example}%
\newtheorem{remark}[theorem]{Remark}%
\numberwithin{theorem}{section}
\newtheorem{lemma}[theorem]{Lemma}
\theoremstyle{thmstylethree}%

\def\R{{\mathbb R}}

\def\F{{\mathcal F}}

\def\argmax{{\rm argmax}}

\begin{document}
\title[Article title]{Accelerated Dinkelbach Method}


\author[a]{\fnm{Hanzhi} \sur{Chen}}\email{BYchz@buaa.edu.cn}

\author[a]{\fnm{Chuyue} \sur{Zheng}}\email{21091003@buaa.edu.cn}

\author*[a]{\fnm{Yong} \sur{Xia}}\email{yxia@buaa.edu.cn}

\affil[a]{\orgdiv{LMIB of the Ministry of Education}, \orgname{School of Mathematical Sciences}, \orgaddress{\street{Beihang University}, \city{Beijing}, \postcode{100191}, \state{}
\country{People's Republic of China}}}

\abstract{
The classical Dinkelbach method (1967) solves fractional programming via a parametric approach, generating a  decreasing upper bound sequence that converges to the optimum. Its important variant, the interval Dinkelbach method (1991), constructs convergent upper and lower bound sequences that bracket the solution and achieve quadratic and superlinear convergence, respectively, under the assumption that the parametric function is twice continuously differentiable. 
However, this paper demonstrates that a minimal correction, applied solely to the upper bound iterate, is sufficient to 
{\color{black}boost the convergence of the method, achieving  superquadratic and cubic rates for the upper and lower bound sequences, respectively.} By strategically integrating this correction, we develop a globally convergent, non-monotone, and accelerated Dinkelbach algorithm—the first of its kind, to our knowledge. Under sufficient differentiability, the new method achieves an asymptotic average convergence order of at least $\sqrt{5}$ per iteration, surpassing the quadratic order of the original algorithm. Crucially, this acceleration is achieved while maintaining the key practicality of solving only a single subproblem per iteration.}




\maketitle
\section{Introduction}
In this paper, we investigate the fractional programming problem
\begin{equation}\label{eq:Fractional_problem}\tag{FP}
\alpha^{*} = \min_{x \in \mathcal{F}} \frac{f_1(x)}{f_2(x)},
\end{equation}
where $f_1 : \mathbb{R}^n \to \mathbb{R}$ and $f_2 : \mathbb{R}^n \to \mathbb{R}$ are continuous functions defined on a compact set $\mathcal{F} \subset \mathbb{R}^n$, with $f_2(x) > 0$ for all $x \in \mathcal{F}$.  
\eqref{eq:Fractional_problem} has at least one optimal solution, denoted by $x^{*}$.

Fractional programming originated in the mid-20th century, emerging from foundational work in economics, operations research, and game theory.
Notably,  John von Neumann's seminal work in 1928 on the minimax theorem \cite{neumann1928zur}  implicitly addressed fractional optimization, providing an important mathematical basis for the field.
The formal development of fractional programming began in 1962 when Charnes and Cooper \cite{charnes1962programming}  introduced linear fractional programming, focusing on the optimization of ratios of linear functions. This was followed by Dinkelbach's landmark contribution \cite{Dinkelbach1967} in 1967, which established the theoretical underpinnings for nonlinear fractional programming and significantly expanded the scope of the field. Since then, fractional programming has garnered sustained research interest due to its broad applicability and mathematical richness.
Comprehensive treatments of fractional programming are available in several influential references. For instance, foundational developments and advanced methodologies can be found in works such as \cite{schaible1983fractional,schaible1995fractional,rodenas1999extensions,frenk2005fractional,stancu2012fractional,bajalinov2013linear}. These contributions have collectively shaped the evolution of fractional programming into a mature and versatile area of optimization theory.

As demonstrated by Jagannathan~\cite{jagannathan1966parametric}, Dinkelbach~\cite{Dinkelbach1967}, and Geoffrion~\cite{geoffrion1967bicriterion}, \eqref{eq:Fractional_problem} is equivalently formulated by identifying the zero of its associated parametric function $g:\mathbb{R} \rightarrow \mathbb{R}$ 
\begin{equation}\label{eq:parametric}
   g(\alpha) = \max_{x \in \mathcal{F}} \left\{ -f_1(x) + \alpha f_2(x) \right\}.
\end{equation}
This function exhibits the following remarkable properties \cite{Dinkelbach1967,schaible1976fractional}.

\begin{itemize}
    \item $g$ is strictly increasing, with a unique root $\alpha^*$ that corresponds to the global optimum of \eqref{eq:Fractional_problem}.

    \item $g$ is convex. For any $x_{0} \in \argmax_{x \in \mathcal{F}} \left\{ -f_1(x) + \alpha_{0} f_2(x) \right\},$ it holds  $f_2(x_{0}) \in \partial g(\alpha_{0}),$ where $\partial g(\alpha_{0})$ denotes the subdifferential of $g$ at $\alpha_{0}$. Furthermore, if $g$ is differentiable at $\alpha_{0}$, then 
$\partial g(\alpha_0)=\{g'(\alpha_0)\}=\{f_{2}(x_{0})\}$.   
\end{itemize}

{\color{black}
 \subsection{Preliminaries}
This subsection establishes the notation, core concepts, and a foundational lemma necessary for the subsequent analysis.

\textbf{Notation.} 
Let \(C^k[a,b]\) denote the space of functions with continuous \(k\)-th derivative on \([a,b]\).
For function \(g:\R\to \R\) and distinct points \(\alpha, \beta \in \mathbb{R}\), the first-order divided difference is defined as
\[
\delta g(\alpha, \beta) := \frac{g(\alpha) - g(\beta)}{\alpha - \beta}.
\]

Consider a sequence \(\{\alpha_k\} \subset \mathbb{R}\) converging to \(\alpha^*\). We characterize its convergence rate using two related concepts.
The sequence is said to converge with order \(s \geq 1\) if there exists a constant \(C > 0\) such that for all \(k\),
\[\frac{|\alpha_{k+1} - \alpha^*|}{|\alpha_k - \alpha^*|^s} \leq C.\]
The sequence is said to converge with asymptotic order \(s \geq 1\) if there exists a constant \(C \geq 0\) such that
\[
\limsup_{k \to \infty} \frac{|\alpha_{k+1} - \alpha^*|}{|\alpha_k - \alpha^*|^s} = C.
\]
 Note that \(C = 0\) is permitted for asymptotic order, indicating particularly fast convergence.
Specific terminology applies based on the order \(s\) and constant \(C\):
\begin{itemize}
    \item Linear convergence: order $s=1$ with \(0 < C < 1\),
    \item Superlinear convergence: asymptotic order $s=1$ with \(C = 0\),
    \item Quadratic convergence: order $s=2$,
    \item Cubic convergence: order $s=3$.
\end{itemize}

We also employ the Landau notation to characterize the asymptotic behavior of various quantities in our analysis.
For sequences $\{a_k\}$ and $\{b_k\}$ converging to zero:
\begin{itemize}
    \item $a_k = O(b_k)$ indicates that there exist constants $C > 0$ and $N > 0$ such that $|a_k| \leq C|b_k|$ for all $k \geq N$.
    \item $a_k = o(b_k)$ indicates that $a_k$ decays faster than $b_k$, satisfying $\lim_{k \to \infty} a_k/b_k = 0$.
\end{itemize}

We now provide a key Lemma on $\delta g$, where $g$ is defined in \eqref{eq:parametric}.
\begin{lemma}\label{lemma:subgradient_property}
For any distinct $\alpha<\beta$, the divided difference $\delta g(\alpha,\beta)$ satisfies the following subgradient bounds:
\begin{align}
    f_2(x_{\alpha}) \leq \delta g(\alpha, \beta) =\delta g(\beta,\alpha)\leq f_2(x_{\beta}),
\end{align}
where
\[
x_{t}\in \argmax_{x\in\F}\{-f_{1}(x)+tf_{2}(x)\},\quad t=\alpha,\,\beta.
\]
\end{lemma}
Since $f_{2}>0$,
the lemma also indicates the following  Corollary.
\begin{corollary}
    \label{coro_supply}
    For any distinct $\alpha,\beta\in\mathbb{R}$, it holds $\delta g(\alpha,\beta)>0$.
\end{corollary}

\subsection{Overview of Dinkelbach method}
The primary approach to solving \( g(\alpha) = 0 \) is based on the generalized  Newton's method, where the derivative is replaced by a subgradient. Isbell and Marlow~\cite{isbell1956attrition} first applied this method to linear fractional programming, and  Dinkelbach~\cite{Dinkelbach1967} later extended it to general nonlinear cases through the following iterative scheme:
\begin{equation}
\label{eq:dinkelbach}
    \alpha_{k+1} = \alpha_k - \frac{g(\alpha_k)}{f_2(x_k)},\quad k=0,1,\cdots,
\end{equation}
where 
$\alpha_{0}=f_{1}(x_{0})/f_{2}(x_{0})$ for a feasible point $x_{0}\in \F$ and 
$x_{k} \in \argmax_{x\in\F}\{-f_{1}(x) + \alpha_{k}f_{2}(x)\}$.
Observing the identity $g(\alpha_k) = -f_1(x_k) + \alpha_k f_2(x_k)$, the update rule \eqref{eq:dinkelbach}  simplifies to
\begin{equation}
\label{eq:simplified}
    \fbox{
    $\alpha_{k+1} = \frac{f_1(x_k)}{f_2(x_k)}, \quad k=0,1,\cdots.$}
\end{equation}
This procedure is now universally known as the Dinkelbach method. A geometric interpretation of the iteration is provided in Figure~\ref{fig:algorithm_1}.
\begin{figure}
\begin{center}
\begin{tikzpicture}
\draw[->] (-4,0) -- (4,0) node[anchor=west]  {$\alpha$};
\draw[domain=-3:2.4]plot(\x,{2^(\x)-1});
\draw[fill=black] (2,3) circle (1.5pt) node[right] {$(\alpha_{k},g(\alpha_{k}))$}; 
\draw[dashed, line width=0.5]
[domain=0.5:2.4]plot(\x,{4*0.693147*(\x-2)+3})node[below right] {$g(\alpha_k)+f_2(x_k)(\alpha-\alpha_k)=0$};
\draw[fill=black] (0,0)circle (1.5pt);
\draw[fill=black] (0.9169787,0) circle (1.5pt)node[below right] {$(\alpha_{k+1},0)$}; 
\draw[dashed, line width=0.5][domain=-1:2]plot(0.9169787,\x);
\draw[fill=black] (0.9169787,0.8881569) circle (1.5pt) node[above left] {$(\alpha_{k+1},g(\alpha_{k+1}))$};
\end{tikzpicture}
\caption{
The iterate $\alpha_{k+1}$ is defined as the root of the tangent line to $g$ at $\alpha_k$.}
\label{fig:algorithm_1}
\end{center}
\end{figure}
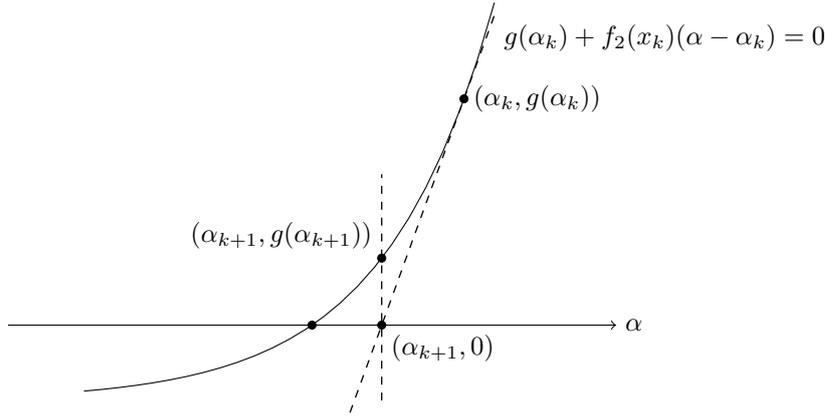

The Dinkelbach method generates a strictly decreasing upper bound sequence {\color{black}$\{\alpha_{k}\}$} that converges to $\alpha^*$.
Subsequent theoretical analyses \cite{crouzeix1985algorithm,flachs1985generalized,borde1987convergence,benadada1993rate,schaible1976fractional} have rigorously established convergence rates under specified regularity conditions. Key known results include:
\begin{itemize}
    \item $\{\alpha_k\}$  achieves superlinear convergence,
    \item 
    The convergence rate becomes quadratic if \(g\) is twice continuously differentiable.
\end{itemize}

We note a point of clarification regarding the convergence rate.
Benadada et al.
\cite{benadada1993rate} report a convergence rate of 1.618 or 2 under certain conditions on 
$f_1$, $f_2$, and $\mathcal{F}$.
 This claim requires clarification. The value 1.618 specifically applies to the generalized Dinkelbach method for multi-objective fractional programming.
In contrast, for the single-objective case, Theorem 4.1 in \cite{borde1987convergence} states that the convergence rate is at least quadratic under the same conditions.
Notably, \cite{borde1987convergence} was published before \cite{benadada1993rate}.  

The method remains widely applied across multiple domains. Recent deployments include \cite{You2009,Jia2020,Azarhava2020,Zhou2021,lu2022dinkelbach,jin2024fractional,wang2023energy}. 
Iteration~\eqref{eq:simplified} can also be directly derived by applying the fixed-point method to solve 
\[
 \alpha^*= \frac{f_1(x^*)}{f_2(x^*)}.
\]
See \cite{beck2009convex,yang2018efficiently} for an in-depth case study involving Tikhonov regularized total least squares.

The escalating scale and complexity of modern optimization problems have intensified the demand for accelerated versions. 
In this vein,  Dadush et al.~\cite{Dadush2023}  proposed a look-ahead Dinkelbach method that reduces computational complexity in linear fractional combinatorial optimization, although it does not improve the theoretical convergence rate in the continuous case. Yang and Xia~\cite{yang2024polynomially} later adapted this look-ahead strategy to unconstrained \(\{-1, 1\}\)-type quadratic optimization.

 \subsection{Overview of interval Dinkelbach method}

Although the standard Dinkelbach method is widely used, its limitation to generating only an upper bound sequence restricts its applicability. Yamamoto et al.~\cite{yamamoto2007efficient}  experimentally demonstrated that hybrid architectures, which integrate the classical Dinkelbach scheme within a two-sided bounding framework, can effectively solve larger-scale convex-convex fractional programming problems, showing significant advantages over the single-bound approach.

Inspired by this finding, 
we review 
an important modification by
Pardalos and Phillips~\cite{Pardalos1991}, which we refer to as the interval Dinkelbach method. This framework iteratively constructs sequences of both upper and lower bounds, enclosing the solution within a contracting interval. The initial interval $[\gamma_0, \alpha_0]$ satisfies $g(\gamma_0) \leq 0 \leq g(\alpha_0)$.

\begin{itemize}
    \item The lower bound sequence $\{\gamma_{k}\}$ is generated by the Secant method:
\begin{equation}\label{eq:lower bound sequence}
    \fbox{
    $\gamma_{k+1}=\gamma_{k}-g(\gamma_{k})\frac{\alpha_{k}-\gamma_{k}}{g(\alpha_{k})-g(\gamma_{k})}, \quad k=0,1,\cdots.$}
    \end{equation}
    
    \item The upper bound sequence $\{\alpha_{k}\}$ is given by the original Dinkelbach update:
\begin{equation}\label{eq:upper bound sequence}
    \fbox{
    $\alpha_{k+1}=\frac{f_{1}(x_{k})}{f_{2}(x_{k})}, \quad x_{k}\in \argmax_{x\in \F}\{-f_{1}(x)+\alpha_{k}f_{2}(x)\}, \quad k=0,1,\cdots.$}
    \end{equation}
\end{itemize} 

The iterative process of the interval Dinkelbach method is illustrated in Figure~\ref{fig:algorithm_2}. Under the assumption that $g$ is twice continuously differentiable on $[\gamma_0, \alpha_0]$, the method exhibits strong convergence properties. The upper bound sequence $\{\alpha_k\}$ converges quadratically, while the lower bound sequence $\{\gamma_k\}$ converges superlinearly, 
with both sequences being monotonic (decreasing and increasing, respectively).

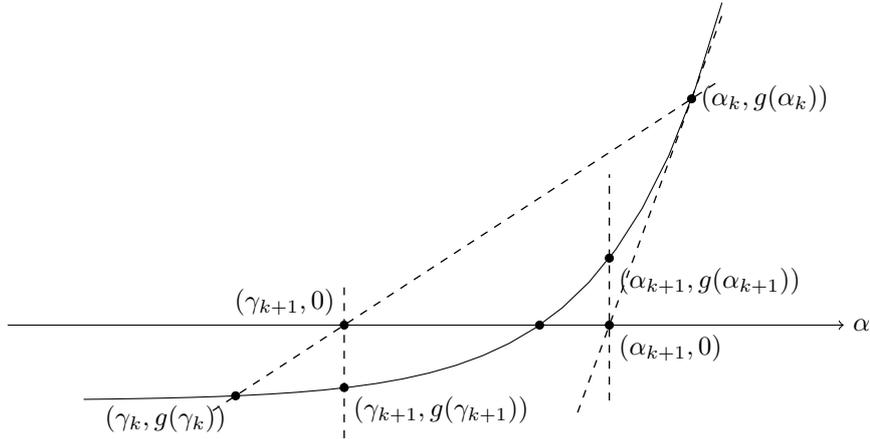
\begin{figure}
\begin{center}
\begin{tikzpicture}
\draw[->] (-7,0) -- (4,0)node[anchor=west]  {$\alpha$};
\draw[domain=-6:2.4]plot(\x,{2^(\x)-1});
\draw[fill=black] (0,0)circle (1.5pt);
\draw[fill=black] (2,3) circle (1.5pt)node[right] {$(\alpha_{k},g(\alpha_{k}))$};
\draw[dashed, line width=0.5][domain=0.5:2.4]plot(\x,{4*0.693147*(\x-2)+3});
\draw[fill=black] (0.9169787,0) circle (1.5pt)node[below right] {$(\alpha_{k+1},0)$}; 
\draw[dashed, line width=0.5][domain=-1:2]plot(0.9169787,\x);
\draw[fill=black] (0.9169787,0.8881569) circle (1.5pt) node[below right] {$(\alpha_{k+1},g(\alpha_{k+1}))$};
\draw[fill=black] (-4,-0.9375) circle (1.5pt) node[below left] {$(\gamma_{k},g(\gamma_{k}))$};
\draw[dashed, line width=0.5][domain=-4.3:2.4]plot(\x,{3.9375*(\x-2)/6+3});
\draw[fill=black] (-2.5714,0) circle (1.5pt) node[above left] {$(\gamma_{k+1},0)$};
\draw[dashed, line width=0.5][domain=-1.5:0.5]plot(-2.5714,\x);
\draw[fill=black] (-2.5714,-0.82867317) circle (1.5pt) node[below right] {$(\gamma_{k+1},g(\gamma_{k+1}))$};
\end{tikzpicture}
\end{center}
\caption{
The interval Dinkelbach method generates  monotonic and convergent upper and lower bound sequences. Specifically, the upper bound sequence $\{\alpha_k\}$ is derived from the original Dinkelbach method~\eqref{eq:upper bound sequence}, while the lower bound sequence $\{\gamma_k\}$ is computed using the Secant method \eqref{eq:lower bound sequence}.}
\label{fig:algorithm_2}
\end{figure}

\subsection{Our contributions}
{\color{black}
The main contributions of this paper can be summarized as follows.}

{\color{black}
First, we present an accelerated interval Dinkelbach method. This method introduces a minor modification to the upper bound iteration of the classical interval Dinkelbach method, while preserving the original lower bound iteration. 
Under the same standard assumption of twice continuous differentiability required by the classical method, the accelerated version achieves  faster convergence rates. Specifically, 
\begin{itemize}
\item
The lower bound sequence
$\{\gamma_k\}$ achieves cubic convergence, a substantial advancement over the superlinear convergence of the conventional method.
    \item 
     The upper bound sequence $\{\alpha_k\}$ attains a superquadratic convergence rate, improving upon its classical quadratic counterpart.
\end{itemize}}

Second, we propose an accelerated Dinkelbach method, built upon the accelerated Newton method introduced by Fernández-Torres~\cite{FernandezTorres2015} and the modification from the accelerated interval Dinkelbach method. Specifically,
\begin{itemize}
    \item  We establish that this method generates a non-monotone yet globally convergent sequence-the first accelerated Dinkelbach algorithm, to our knowledge, possessing both properties.
    \item 
    Our method achieves an  asymptotic average convergence order of $\sqrt{5}\approx 2.236$, under sufficient differentiability assumptions, outperforming the original Dinkelbach method's quadratic order.
    \item 
Furthermore, the convergence behavior exhibits a periodic structure depending on the higher-order derivatives of $g$ at $\alpha^*$, leading to even higher asymptotic average orders in specific cases.
\end{itemize}

\subsection{Organization}
The remainder of the paper is organized as follows. Section~\ref{section: Modified Dinkelbach} develops the accelerated version of the interval Dinkelbach method along with its convergence analysis. Section~\ref{section:Accelerated dinkelbach method} develops the accelerated Dinkelbach method and conducts its convergence analysis.
{\color{black}
Section~\ref{section:Conclusion} summarizes the main contributions of this work, outlines promising future research directions, and presents two open problems.}

\section{Accelerated interval  Dinkelbach method}\label{section: Modified Dinkelbach}

In this section, we propose an accelerated interval Dinkelbach method. The motivation for this acceleration stems from two key observations.

First, the upper bound iteration in the classical method, given by $\alpha_{k+1} = {f_1(x_k^{\alpha})}/{f_2(x_k^{\alpha})}$, depends exclusively on $\alpha_k$ and does not utilize the information from the lower bound sequence $\{\gamma_k\}$. This represents a potential source of inefficiency.

Second, we note that the computation of $g(\gamma_k)$ via the subproblem~\eqref{eq:parametric}  inherently yields the maximizer $x_k^{\gamma}$, and consequently, the values $f_1(x_k^{\gamma})$ and $f_2(x_k^{\gamma})$ are obtained as natural byproducts. Crucially, the ratio ${f_1(x_k^{\gamma})}/{f_2(x_k^{\gamma})}$ provides a new estimate for the root $\alpha^*$.  The accelerated method leverages this underutilized information by strategically incorporating the ratio into the upper bound iteration.

We now present the accelerated interval Dinkelbach method.

\begin{itemize}
    \item 
    The lower bound sequence $\{\gamma_k\}$ is obtained using the Secant method, the same as the original interval Dinkelbach method.
\begin{equation}\label{eq:lower bound sequence next}
\fbox{$\gamma_{k+1}=\gamma_{k}-g(\gamma_{k})\frac{\alpha_{k}-\gamma_{k}}{g(\alpha_{k})-g(\gamma_{k})}, \quad k=0,1,\cdots.$}
    \end{equation}
    \item 
    The upper bound sequence $\{\alpha_k\}$  is determined by finding the minimum zero-crossing of the two tangent lines constructed at the current lower and upper bounds, deviating from conventional methods that use only the upper bound tangent lines zero-crossing.
\begin{equation}\label{eq:upper bound sequence next}
\fbox{$\alpha_{k+1}=\min\left\{ \alpha_{k}-\frac{g(\alpha_{k})}{f_{2}(x_{k}^{\alpha})}, \gamma_{k+1}-\frac{g(\gamma_{k+1})}{f_{2}(x_{k+1}^{\gamma})}\right\}, \quad k=0,1,\cdots,$}
    \end{equation}
    \noindent where
    \[x_{k+1}^{\gamma}\in \argmax_{x\in\F}\{-f_{1}(x)+\gamma_{k+1}f_{2}(x)\},\quad x_{k}^{\alpha}\in \argmax_{x\in\F}\{-f_{1}(x)+\alpha_{k}f_{2}(x)\}.\]
\end{itemize} 

Figure~\ref{fig:algorithm 3} provides a schematic overview of the algorithm's workflow.

\begin{figure}[!h]
\begin{center}
\begin{tikzpicture}
\draw[->] (-4,0) -- (4,0) node[anchor=west] {$\alpha$};
\draw[domain=-4.5:4]plot(\x,{1.5^(\x)-1});
\draw[fill=black] (0,0)circle (1.5pt);
\draw[fill=black] (3,2.375) circle (1.5pt)node[right] {$(\alpha_{k},g(\alpha_{k}))$};
\draw[fill=black] (-2,-0.555555) circle (1.5pt) node[above left] {$(\gamma_{k},g(\gamma_{k}))$};
\draw[dashed, line width=0.5][domain=-2.6:3.6]plot(\x,{(0.475+0.11111)*(\x-3)+2.375});
\draw[dashed, line width=0.5][domain=-1:0.5]plot(-1.05213278,\x);

\draw[fill=black] (-1.05213278,-0.347277454) circle (1.5pt) node[below right] {$(\gamma_{k+1},g(\gamma_{k+1}))$};
\draw[dashed, line width=0.5][domain=1.1:4]plot(\x,{1.368444739*(\x-3)+2.375});
\draw[dashed, line width=0.5][domain=-3:2]plot(\x,{0.264656147*(\x+1.05213278)-0.347277454});
\draw[fill=black] (0.26,0) circle (1.5pt)node[below right] {$(\alpha_{k+1},0)$};;
\end{tikzpicture}
\caption{The accelerated interval Dinkelbach method  {\color{black} only modifies} the upper bound iterations, as shown in \eqref{eq:upper bound sequence next}. In the illustrated specific case, the zero-crossing of the tangent line to $g$ at $\gamma_{k+1}$ may provide a better approximation to $\alpha^{*}$ than that from the tangent at $\alpha_{k}$.}
\label{fig:algorithm 3}
\end{center}
\end{figure}
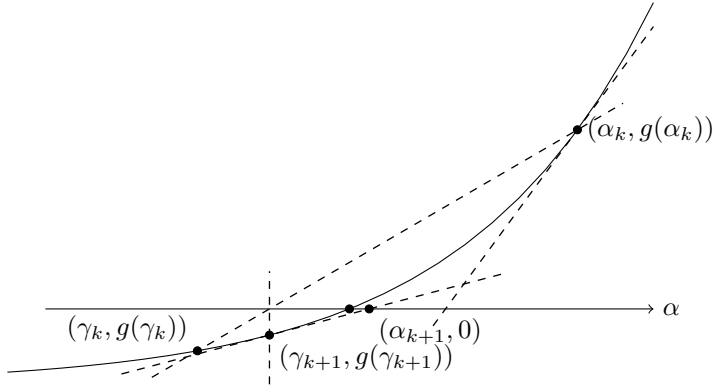  
The iteration~\eqref{eq:upper bound sequence next} can also be simplified as
\begin{equation}
\label{eq:10}
\alpha_{k+1}=\min\left\{ \frac{f_{1}(x_{k}^{\alpha})}{f_{2}(x_{k}^{\alpha})}, \frac{f_{1}(x_{k+1}^{\gamma})}{f_{2}(x_{k+1}^{\gamma})}\right\}.
\end{equation}

Critically, our approach introduces no additional subproblems
during iterations
beyond the original interval method. 
This is achieved by reusing the solution $x^{\gamma}_{k+1}$ already 
obtained when computing $g(\gamma_{k+1})$.
Surprisingly, this subtle adjustment leads to significant improvements, as demonstrated in subsequent key findings.

\begin{lemma}\label{lemma:inequ_1}
    The sequences generated by the accelerated interval Dinkelbach method satisfy the following properties.
    \begin{itemize}
        \item  $\{\alpha_k\}$ is monotonically decreasing and $g(\alpha_{k})\ge 0$ for $k=0,1,\cdots$.
\item $\{\gamma_k\}$ is monotonically increasing and $g(\gamma_{k})\le 0$  for $k=0,1,\cdots$.
    \end{itemize}
    \noindent Equivalently, for $k = 0, 1, \dots$, the inequality below holds:
    \begin{equation}\nonumber
        \gamma_k \leq \gamma_{k+1} \leq \alpha^* \leq \alpha_{k+1} \leq \alpha_k.
    \end{equation}
\end{lemma}
   
\begin{proof}
    We first prove $\alpha_{k}\ge \alpha^{*}$ for all $k$.
    For $k=0$,
    $\alpha_0\geq \alpha^*$.
    For $\forall k\ge 1$, since $f_{2}(x_{k}^{\gamma})\in \partial g(\gamma^{k})$, $f_{2}(x_{k-1}^{\alpha})\in \partial g(\alpha^{k-1})$, we have 
\begin{align}
    g(\gamma_{k})-g(\alpha^{*})&\le f_{2}(x_{k}^{\gamma})(\gamma_{k}-\alpha^{*}),\nonumber\\
 g(\alpha_{k-1})-g(\alpha^{*})&\le f_{2}(x_{k-1}^{\alpha})(\alpha_{k-1}-\alpha^{*}).\nonumber
\end{align}
As $g(\alpha^{*}) = 0$ and $f_2> 0$, it follows from the above inequalities that
\begin{align}
        \gamma_{k}-\frac{g(\gamma_{k})}{f_{2}(x_{k}^{\gamma})}&\ge \alpha^{*},\label{eq:gammak+1}\\
        \alpha_{k-1}-\frac{g(\alpha_{k-1})}{f_{2}(x_{k-1}^{\alpha})}&\ge \alpha^{*}.\label{eq:inequ 2}
\end{align}
Substituting
\eqref{eq:gammak+1} and \eqref{eq:inequ 2} into \eqref{eq:upper bound sequence next}, we confirm
\begin{equation}
\nonumber
\alpha_{k}= \min\left\{\alpha_{k-1}-\frac{g(\alpha_{k-1})}{f_{2}(x_{k-1}^{\alpha})},\gamma_{k}-\frac{g(\gamma_{k})}{f_{2}(x_{k}^{\gamma})}\right\}\ge \alpha^{*},
\end{equation}
Moreover, by the non-negativity of $g(\alpha_k)$ and $f_2(x_k^\alpha)$,  
it holds $\alpha_{k+1} \leq \alpha_k$ for all $k\ge 0$.
The monotonicity of the lower bound sequence has been established in \cite{Pardalos1991}, 
which completes the proof.
\end{proof}
{\color{black}
As an immediate consequence of Lemma~\ref{lemma:inequ_1} and the monotone convergence theorem, we obtain the convergence of both bound sequences.}
\begin{theorem}
\label{theorem: converge}
   $\{\gamma_{k}\}$ and $\{\alpha_{k}\}$ both converge to $\alpha^{*}$.
\end{theorem}
\begin{proof}
Since the sequences $\{\alpha_k\}$ and $\{\gamma_k\}$ are both monotonic and bounded, by the Monotone Convergence Theorem, their limits $\lim_{k\rightarrow\infty}\gamma_k := \gamma$ and $\lim_{k\rightarrow\infty}\alpha_k := \alpha$ exist. Moreover,  the continuity of $g$ {\color{black}(directly derived from the convexity)} implies that $\lim_{k\rightarrow\infty}g(\gamma_k) = g(\gamma)$ and $\lim_{k\rightarrow\infty}g(\alpha_k) = g(\alpha)$.
Taking the limit as $k\rightarrow\infty$ in \eqref{eq:lower bound sequence next} and \eqref{eq:upper bound sequence next}, we derive
\begin{equation}
\label{eq:converge gamma}
\gamma = \gamma + g(\gamma)\frac{\alpha - \gamma}{g(\alpha) - g(\gamma)} = \gamma + \frac{g(\gamma)}{\delta g(\gamma,\alpha)},
\end{equation}
\begin{equation}
\label{eq:converge alpha}
\alpha =\lim_{k\rightarrow\infty}\min\left\{\alpha_{k-1}-\frac{g(\alpha_{k-1})}{f_{2}(x_{k-1}^{\alpha})},\gamma_{k}-\frac{g(\gamma_{k})}{f_{2}(x_{k}^{\gamma})}\right\}
\le \alpha-\frac{g(\alpha)}{\max_{\xi\in\mathcal{F}}f_{2}(\xi)}
\end{equation}
Since $\delta g(\gamma,\alpha)$ is bounded, it follows from \eqref{eq:converge gamma} that $g(\gamma) = 0$. {\color{black}Furthermore}, noting that $g(\alpha) = \lim_{k\rightarrow\infty}g(\alpha_k) \ge 0$ and $f_2> 0$, we conclude from \eqref{eq:converge alpha} that $g(\alpha) = 0$. Given that $g$ is strictly increasing with a unique zero at $\alpha^*$, this implies $\alpha = \gamma = \alpha^*$, thus completing the proof of convergence.
\end{proof}

Modification \eqref{eq:upper bound sequence next}
enhances the convergence rate of the lower bound sequence from superlinear to cubic under the same regularity conditions, as formalized below. 
{\color{black}
For the remainder of this section,} we assume that the sequences \(\{\alpha_k\}\) and \(\{\gamma_k\}\) are infinite. This infinitude ensures $\gamma_{k}<\alpha^{*}<\alpha_{k}$ for each $k$.
\begin{theorem}\label{theorem:enhanced interval  Dinkelbach}
{\color{black}If $g\in C^{2}[\gamma_{0},\alpha_{0}]$,} then the sequence $\{\gamma_{k}\}$ generated by \eqref{eq:lower bound sequence next} converges to $\alpha^{*}$ with at least cubic convergence.
\end{theorem}
\begin{proof} 
 Firstly, for any given $k\ge 0$, we estimate $\gamma_{k+1}-\alpha^{*}$. Since {\color{black}$g\in C^{2}[\gamma_0,\alpha_{0}]$}, there exist $\xi_{k}\in[\gamma_{k},\alpha^{*}]$ and $\zeta_{k}\in[\alpha^{*},\alpha_{k}]$  such that
\begin{align}
g(\gamma_{k})&=g(\alpha^{*})+(\gamma_{k}-\alpha^{*})g'(\alpha^{*})+\frac{1}{2}(\gamma_{k}-\alpha^{*})^{2}g''(\xi_{k}),\label{eq: Taylor expand at gammak}\\
g(\alpha_{k})&=g(\alpha^{*})+(\alpha_{k}-\alpha^{*})g'(\alpha^{*})+\frac{1}{2}(\alpha_{k}-\alpha^{*})^{2}g''(\zeta_{k}).\label{eq: Taylor expand at alphak}
\end{align}
Since $g(\alpha^{*}) = 0$, we compute the divided differences $\delta g(\gamma_{k},\alpha^{*})$ and $\delta g(\gamma_{k},\alpha_{k})$ through  \eqref{eq: Taylor expand at gammak} and \eqref{eq: Taylor expand at alphak}:
\begin{align}
\delta g(\gamma_{k},\alpha^{*}) &= \frac{g(\gamma_{k}) - g(\alpha^{*})}{\gamma_{k} - \alpha^{*}} = g'(\alpha^{*}) + \frac{1}{2}(\gamma_{k} - \alpha^{*})g''(\xi_{k}),\nonumber\\
\delta g(\gamma_{k},\alpha_{k}) &= \frac{g(\gamma_{k}) - g(\alpha_{k})}{\gamma_{k} - \alpha_{k}} = g'(\alpha^{*}) + \frac{1}{2}\frac{(\gamma_{k} - \alpha^{*})^{2}}{\gamma_{k} - \alpha_{k}}g''(\xi_{k}) - \frac{1}{2}\frac{(\alpha_{k} - \alpha^{*})^{2}}{\gamma_{k} - \alpha_{k}}g''(\zeta_{k}).\nonumber
\end{align}
Subtracting these two expressions yields
\begin{align}
\delta g(\gamma_k,\alpha_{k}) - \delta g(\gamma_k,\alpha^{*}) &= \frac{1}{2}\frac{(\alpha_{k} - \alpha^{*})(\gamma_{k} - \alpha^{*})}{\gamma_{k} - \alpha_{k}}g''(\xi_{k}) - \frac{1}{2}\frac{(\alpha_{k} - \alpha^{*})^{2}}{\gamma_{k} - \alpha_{k}}g''(\zeta_{k}) \nonumber \\
&= \frac{\alpha_{k} - \alpha^{*}}{2}\left(\frac{\alpha^{*} - \gamma_{k}}{\alpha_{k} - \gamma_{k}}g''(\xi_{k}) + \frac{\alpha_{k} - \alpha^{*}}{\alpha_{k} - \gamma_{k}}g''(\zeta_{k})\right).\nonumber
\end{align}
{\color{black}Since
 $\gamma_k < \alpha^* < \alpha_k$, we have}
\[
\frac{\alpha^* - \gamma_k}{\alpha_k - \gamma_k} > 0 \quad \text{and} \quad \frac{\alpha_k - \alpha^*}{\alpha_k - \gamma_k} > 0.
\]
Therefore, we establish the following bound:
\begin{align}
|\delta g(\gamma_k,\alpha_{k}) - \delta g(\gamma_k,\alpha^{*})| &= \frac{(\alpha_{k} - \alpha^{*})}{2}\left|\frac{\alpha^{*} - \gamma_{k}}{\alpha_{k} - \gamma_{k}}g''(\xi_{k}) + \frac{\alpha_{k} - \alpha^{*}}{\alpha_{k} - \gamma_{k}}g''(\zeta_{k})\right| \nonumber \\
&\leq \frac{(\alpha_{k} - \alpha^{*})}{2}\left(\frac{\alpha^{*} - \gamma_{k}}{\alpha_{k} - \gamma_{k}}|g''(\xi_{k})| + \frac{\alpha_{k} - \alpha^{*}}{\alpha_{k} - \gamma_{k}}|g''(\zeta_{k})|\right) \nonumber \\
&\leq \frac{(\alpha_{k} - \alpha^{*})}{2}\max_{\zeta\in[\gamma_{k},\alpha_{k}]}|g''(\zeta)|\left(\frac{\alpha^{*} - \gamma_{k}}{\alpha_{k} - \gamma_{k}} + \frac{\alpha_{k} - \alpha^{*}}{\alpha_{k} - \gamma_{k}}\right) \nonumber \\
&= \frac{\sigma_{k}}{2}(\alpha_{k} - \alpha^{*}),\label{eq:taylor expand}
\end{align}
where 
\begin{equation}
\label{eq:sigma}
\sigma_{k}=\max_{\zeta\in[\gamma_{k},\alpha_{k}]}|g''(\zeta)|
\end{equation}
exists since 
$g''\in C[\gamma_{0},\alpha_{0}]$.
Given $\alpha_{k}>\alpha^{*}>\gamma_{k}\ge \gamma_{0}$, {\color{black}Lemma \ref{lemma:subgradient_property} implies
\begin{equation}
    \label{eq:456}
\delta g(\gamma_{k},\alpha_{k})\ge f_{2}(x_{k}^{\gamma})\ge f_{2}(x_{0}^{\gamma}),
\ \text{and}\ f_{2}(x^{*})\ge \delta g(\gamma_{k},\alpha^{*}).
\end{equation}}
Given that $g(\alpha^{*})=0$, we have an upper estimate on $|\gamma_{k+1}-\alpha^{*}|$:
\begin{align}
|\gamma_{k+1}-\alpha^{*}|&=\left|\gamma_{k}-\alpha^{*}-(g(\gamma_{k})-g(\alpha^{*}))\frac{\gamma_{k}-\alpha_{k}}{g(\gamma_{k})-g(\alpha_{k})}\right|\nonumber\\
&=(\alpha^{*}-\gamma_{k})\left|1-\frac{\delta g(\gamma_{k},\alpha^{*})}{\delta g(\gamma_{k},\alpha_{k})}\right|\nonumber\\
&=(\alpha^{*}-\gamma_{k})\frac{|\delta g(\gamma_{k},\alpha_{k})-\delta g(\gamma_{k},\alpha^{*})|}{\delta g(\gamma_{k},\alpha_{k})}.\nonumber\\
    &\le (\alpha^{*}-\gamma_{k})(\alpha_k - \alpha^*)\frac{\sigma_{k}}{2\delta g(\gamma_{k}.\alpha_{k})}  \nonumber \ (\text{by \eqref{eq:taylor expand}})\\
    &\le \frac{\sigma_{k}}{2f_{2}(x_{0}^{\gamma})} |\gamma_k - \alpha^*|(\alpha_k - \alpha^*)\ (\text{by \eqref{eq:456}})\label{eq:final}.
\end{align}

Secondly, we 
establish an upper bound for $\alpha_k - \alpha^*$.
{\color{black}
By the construction of $\alpha_{k}$ \eqref{eq:upper bound sequence next}, we have 
\begin{align}
\alpha_{k}-\alpha^{*}&\le \gamma_{k}-\frac{g(\gamma_{k})}{f_{2}(x_{k}^{\gamma})}-\alpha^{*}\nonumber\\
&=(\alpha^{*}-\gamma_{k})\left(\frac{\delta g(\gamma_{k},\alpha^{*})-f_{2}(x_{k}^{\gamma})}{f_{2}(x_{k}^{\gamma})}\right)\nonumber\\
&\le(\alpha^* - \gamma_k) \frac{f_2(x^*) - f_2(x_k^\gamma)}{f_2(x_0^\gamma)}\ (\text{by \eqref{eq:456}})\nonumber\\
&\le   \frac{\sigma_{k}}{f_2(x_0^\gamma)} (\alpha^* - \gamma_k)^2,\label{eq:establish}
\end{align}
where \eqref{eq:establish} holds since 
\[
    f_2(x^*) - f_2(x_k^\gamma) = g'(\alpha^*) - g'(\gamma_k) = g''(\omega_k)(\alpha^* - \gamma_k),
\]
for some $\omega_k \in [\gamma_k, \alpha^*]$, and 
\[
g''(\omega_{k})\le \sigma_{k}\]
by the definition of $\sigma_{k}$ in \eqref{eq:sigma}.

Finally, substituting 
the upper bound of $\alpha_{k}-\alpha^{*}$ \eqref{eq:establish} into the upper bound of $|\gamma_{k+1}-\alpha^{*}|$ 
\eqref{eq:final}  yields the cubic convergence}
\[
|\gamma_{k+1}-\alpha^{*}|\le \frac{\sigma_{k}^{2}}{2f_{2}^{2}(x_{0}^{\gamma})}|\gamma_{k}-\alpha^{*}|^{3}\le \frac{\sigma_{0}^{2}}{2f_{2}^{2}(x_{0}^{\gamma})}|\gamma_{k}-\alpha^{*}|^{3},
\]
where $\sigma_{k}= \max_{\zeta \in [\gamma_k, \alpha_k]} |g''(\zeta)|\le  \max_{\zeta \in [\gamma_0, \alpha_0]} |g''(\zeta)|=\sigma_{0}$.
This completes the proof.
\end{proof}
{\color{black}
We now demonstrate the superquadratic convergence of the upper bound sequence $\{\alpha_{k}\}$.
\begin{theorem}
    \label{theorem:enhanced interval Dinkelbach 1}
If $g\in C^{2}[\gamma_{0},\alpha_{0}]$, 
then the upper bound sequence \(\{\alpha_k\}\) converges with superquadratic order.
\end{theorem}
\begin{proof}
For all $k\ge 0$, we obtain
\[
\alpha_{k+1}-\alpha^{*}\le \frac{\sigma_{k+1}}{f_{2}(x_{0}^{\gamma})}(\alpha^{*}-\gamma_{k+1})^{2}\le  \frac{\sigma_{k+1}}{f_{2}(x_{0}^{\gamma})}\frac{\sigma_{k}^{2}}{4f_{2}^{2}(x_{0}^{\gamma})} |\gamma_{k} - \alpha^*|^{2}(\alpha_{k} - \alpha^*)^{2},\]
where the first inequality follows from applying \eqref{eq:establish} to iteration $k+1$ and the second from \eqref{eq:final}. Since \(\gamma_k \to \alpha^*\) as \(k \to \infty\), and the coefficient
\(
\sigma_{k+1}\sigma_k^2/4f_2^3(x_0^\gamma)
\)
is bounded above by \(\sigma_0^3/(4f_2^3(x_0^\gamma))\) (where \(\sigma_0 = \max_{\zeta \in [\gamma_0, \alpha_0]} |g''(\zeta)|\)), we conclude that
\[
\lim_{k \to \infty} \frac{\alpha_{k+1} - \alpha^*}{(\alpha_k - \alpha^*)^2} = 0.
\]
This completes the proof.
\end{proof}}
{\color{black}
\begin{remark}
    Under the same smoothness assumptions, the upper bound sequence attains a higher asymptotic convergence order. Specifically, the order is 3 if \( g''(\alpha^*) \neq 0 \), and is at least \( 3 - \epsilon \) for any \( \epsilon > 0 \) if \( g''(\alpha^*) = 0 \). A detailed proof establishing these rates is possible but involves lengthy technical arguments,  it is therefore omitted for brevity.
\end{remark}}
We conclude this section with a numerical example that validates the convergence orders established in Theorems~\ref{theorem:enhanced interval Dinkelbach} and \ref{theorem:enhanced interval Dinkelbach 1}.
\begin{table}[h]
\centering
\caption{Comparison of convergence behavior between
accelerated and original interval Dinkelbach methods}
\label{tab:comparison}
\begin{tabular}{@{}c rr rr@{}}
\toprule
 & \multicolumn{2}{c}{Accelerated} & \multicolumn{2}{c}{Original} \\
\cmidrule(lr){2-3} \cmidrule(lr){4-5}
\(k\) & \(g(\gamma_{k})\) & \(g(\alpha_{k})\) & \(g(\gamma_{k})\) & \(g(\alpha_{k})\) \\
\midrule
0 & \(-1.86 \times 10^{1}\) & \(1.89 \times 10^{2}\) & \(-1.86 \times 10^{1}\) & \(1.89 \times 10^{2}\) \\
1 & \(-1.30 \times 10^{1}\) & \(7.39 \times 10^{-1}\) & \(-1.30 \times 10^{1}\) & \(7.39 \times 10^{1}\) \\
2 & \(-5.50 \times 10^{-2}\) & \(2.10 \times 10^{-5}\) & \(-6.06 \times 10^{0}\) & \(2.72 \times 10^{1}\) \\
3 & \(-8.05 \times 10^{-9}\) & \(4.52 \times 10^{-19}\) & \(-1.29 \times 10^{0}\) & \(6.32 \times 10^{0}\) \\
4 & \(-2.54 \times 10^{-29}\) & \(4.48\times 10^{-60}\)               & \(-6.06 \times 10^{-2}\) & \(3.21 \times 10^{-1}\) \\
5 &\(-7.92 \times 10^{-91}\) & \(4.38\times 10^{-183}\)   & \(-1.36 \times 10^{-4}\) & \(7.26 \times 10^{-4}\) \\
6 &        \multicolumn{2}{c}{—}            & \(-6.89 \times 10^{-10}\) & \(3.68 \times 10^{-9}\) \\
7 &  \multicolumn{2}{c}{—}                    & \(-1.77 \times 10^{-20}\) & \(9.42 \times 10^{-20}\) \\
8 & \multicolumn{2}{c}{—}                     & \(-1.16\times 10^{-41}\)                    & \(6.19 \times 10^{-41}\) \\
9 & \multicolumn{2}{c}{—}                     & \(-5.01\times 10^{-84}\)                    & \(2.67 \times 10^{-83}\) \\
10 & \multicolumn{2}{c}{—}                     & \(-9.31\times 10^{-169}\)                    & \(4.97 \times 10^{-168}\) \\
\bottomrule
\end{tabular}
\end{table}

\begin{example}
Table~\ref{tab:comparison} provides a comparison between the accelerated interval Dinkelbach method and the original version for the following function. \noindent
For the sake of simplification, we directly define $g$ as 
\[
g(\alpha) = e^{\frac{1}{2}\alpha} + 5\alpha - 9
\]
\noindent with initial values $\alpha_0 = 10$ and $\gamma_0 = -2$. The computations are performed in Mathematica 12.3.1.0 with 200 digit precision. The convergence criterion is triggered when ${\color{black}\min}\{|g(\alpha_k)|, |g(\gamma_k)|\} < 10^{-160}$.

{\color{black}
The convergence orders can be analyzed through the error bounds derived from Lemma~\ref{lemma:subgradient_property}. Since $g(\alpha^{*})=0$, we have
\begin{align*}
f_2(x^{*})(\alpha_k - \alpha^{*}) &\leq g(\alpha_k) \leq f_2(x_0^{\alpha})(\alpha_k - \alpha^{*}), \\
f_2(x_0^{\gamma})(\alpha^{*} - \gamma_k) &\leq -g(\gamma_k) \leq f_2(x^{*})(\alpha^{*} - \gamma_k).
\end{align*}
These inequalities imply that 
$g(\alpha_{k})=O(\alpha_{k}-\alpha^{*})$ and $g(\gamma_{k})=O(\gamma_{k}-\alpha^{*})$, i.e., the convergence rates of $\{\alpha_{k}\}$ and $\{\gamma_{k}\}$ are equivalent to those of $\{g(\alpha_{k})\}$ and $g(\gamma_{k})$, respectively.

The numerical result demonstrates the superquadratic convergence for $\{\alpha_{k}\}$ and cubic convergence $\{\gamma_{k}\}$.}
\end{example}
\section{Accelerated Dinkelbach method}\label{section:Accelerated dinkelbach method}
This section introduces an accelerated Dinkelbach method. We leverage two prior iterates to enhance step expansion and systematically incorporate non-monotonic iterates generated during the convergence process.
\subsection{Motivation}
Motivated by Dinkelbach's application of the generalized Newton method to solve \eqref{eq:parametric}, we take an accelerated Newton-type method as the foundation for our approach.
Consider a sufficiently differentiable, increasing, and convex function $f: \R \to \R$. 
We aim to compute a simple zero of $f$ i.e., a solution 
$x^{*}$ satisfying
 $f(x^{*}) = 0
$ and $f'(x^{*}) \neq 0$.

Fernández-Torres in \cite{FernandezTorres2015}
introduces an accelerated Newton method that attains a local convergence rate of $1 + \sqrt{2}$.
For two previous iterates $x_{k-1}>x_{k}>x^{*}$, Lemma~\ref{lemma:subgradient_property} implies 
\begin{equation}\nonumber
f'(x_k) \leq \delta f(x_k, x_{k-1}).
\end{equation}
\noindent The next iteration $x_{k+1}$ is defined as the root of the secant line passing through $(x_k, f(x_k))$ and $(x_{k-1}, f(x_{k-1})f'(x_{k})/\delta f(x_{k},x_{k-1}))$. Concretely, it takes the form
\begin{equation}\label{eq:Motivation}
    \fbox{$
    x_{k+1}=x_{k}-\frac{f(x_{k})(f(x_{k})-f(x_{k-1}))(x_{k}-x_{k-1})}{f(x_{k})(f(x_{k})-f(x_{k-1}))-f(x_{k-1})f'(x_{k})(x_{k}-x_{k-1})}.$}
\end{equation} 
Figure~\ref{fig:main_result} provides a geometric illustration of this accelerated Newton iteration.
 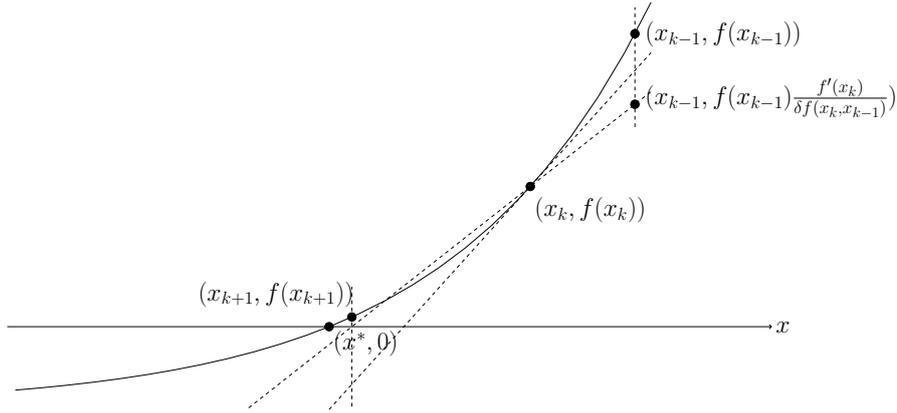
\begin{figure}[ht]
    \centering
    \begin{adjustbox}{width=0.9\textwidth}
    \begin{tikzpicture}[scale=2.3, font=\huge] 
        \draw[->] (-4,0) -- (5.5,0) node[anchor=west]  {$x$}; 
        \draw[domain=-3.9:4] plot (\x,{1.5^(\x)-1});
        \draw[fill=black] (0,0) circle (1.5pt) node[ below right] {$(x^{*},0)$}; 
        \draw[fill=black] (2.5,1.75567596) circle (1.5pt) node[below right, yshift=-5pt] {$(x_{k},f(x_{k}))$};
        \draw[fill=black] (3.8,3.66817130) circle (1.5pt) node[right, xshift=5pt] {$(x_{k-1},f(x_{k-1}))$}; 
         \draw[dashed, line width=0.5] plot[domain=2.5:4] (3.8,\x);
        \draw[fill=black] (3.8,2.785277) circle (1.5pt) node[right, xshift=5pt, yshift=5pt] {$(x_{k-1},f(x_{k-1})\frac{f'(x_{k})}{\delta f (x_{k},x_{k-1})})$};
        \draw[dashed, line width=0.5] plot[domain=0:4] (\x,{1.11733045099*(\x-2.5)+1.75567596});
        \draw[dashed, line width=0.5] plot[domain=-1:4] (\x,{0.7920008*(\x-2.5)+1.75567596});
        \draw[dashed, line width=0.5] plot[domain=-1:0.5] (0.283239739,\x);
        \draw[fill=black] (0.283239739,0.121698249841035) circle (1.5pt) node[above left, xshift=5pt, yshift=5pt] {$(x_{k+1},f(x_{k+1}))$}; 
    \end{tikzpicture}
    \end{adjustbox}
    \caption{
The accelerated method provides more accurate estimate of $x 
_{k+1}$ than the standard Newton method near the optimal solution $x^{*}$.}
\label{fig:main_result} 
\end{figure}
 \begin{theorem}[\cite{FernandezTorres2015}]
   \label{lemma:1+sqrt2}
   Let $f:\mathbb{R} \to \mathbb{R} $ be a sufficiently differentiable function with a root $r$ and $f'(x)\neq 0$ for all $x$. If $x_{-1}$ is sufficiently close to $r$, $x_{0}=x_{-1}-f(x_{-1})/f'(x_{-1})$, then the sequence generated by the iteration~\eqref{eq:Motivation}  converges to $x^*$ with an order of $1 + \sqrt{2}$. 
 \end{theorem}
\begin{remark}
The accelerated Newton method \eqref{eq:Motivation} demonstrates local convergence properties while lacking global convergence guarantees.
\end{remark}
\subsection{Accelerated Dinkelbach method}
\begin{algorithm}
\caption{Accelerated Dinkelbach algorithm}
\label{Algorithm: Accelerated Dinkelbach}
\begin{algorithmic}[1]
\Require $x_{-1}\in\F$.
\Ensure $\min_{x\in\F} f_{1}(x)/f_{2}(x)$ 
\State{$\textbf{Compute}\,\alpha_{-1}=f_{1}(x_{-1})/f_{2}(x_{-1})$ \textbf{and}
$\alpha_{0}=\alpha_{-1}-g(\alpha_{-1})/f_{2}(x_{-1})$. \textbf{Set} $\rho>1$, {\color{black}$\epsilon\ge 0$}}, $k=0$.
\While{$\lvert g(\alpha_{k}) \rvert> \epsilon$}
\If{$g(\alpha_{k-1})f_{2}(x_{k})(\alpha_{k}-\alpha_{k-1})\le \rho g(\alpha_{k})(g(\alpha_{k})-g(\alpha_{k-1}))$}
\begin{equation}
\label{eq:alpha sequence}
\,\quad\,\alpha_{k+1} \leftarrow
    \alpha_{k}-\frac{g(\alpha_{k})(g(\alpha_{k})-g(\alpha_{k-1}))(\alpha_{k}-\alpha_{k-1})}{g(\alpha_{k})(g(\alpha_{k})-g(\alpha_{k-1}))-g(\alpha_{k-1})f_{2}(x_k)(\alpha_{k}-\alpha_{k-1})}\end{equation}
\If{$g(\alpha_{k+1})<0$}
\begin{align}\label{eq:min tangent}&\alpha_{k+2}\leftarrow \min\left\{\alpha_{k+1}-\frac{g(\alpha_{k+1})}{f_{2}(x_{k+1})},\alpha_{k}-\frac{g(\alpha_{k})}{f_{2}(x_{k})}\right\}\\
        &k\leftarrow k+2\nonumber
    \end{align}
    \Else
        \State{\qquad \quad \,$k\leftarrow k+1$}
    \EndIf
    \Else
\begin{flalign}\label{eq:original Dinkelbach}&\qquad\quad \,\alpha_{k+1}\leftarrow \alpha_{k}-\frac{g(\alpha_{k})}{f_{2}(x_{k})}&\\
&\qquad\quad\, k\leftarrow k+1&\nonumber
\end{flalign}
\EndIf
\EndWhile
\end{algorithmic}
\end{algorithm}
In this subsection, we apply the accelerated Newton iteration~\eqref{eq:Motivation} to solve \eqref{eq:Fractional_problem}. 
Since $g$ defined in \eqref{eq:parametric} may be nondifferentiable, we generalize the iteration by substituting the derivative with a subgradient.
Furthermore, we incorporate safeguards to ensure global convergence. The algorithm is formally stated in Algorithm~\ref{Algorithm: Accelerated Dinkelbach}, 
followed by a discussion of key observations.

{\color{black}
In contrast to the classical Dinkelbach method (which generates a monotonic sequence), the accelerated Dinkelbach method exhibits \textbf{non-monotonic behavior}.
Specifically,
a key proposition regarding this property is stated below.}
    \begin{proposition}
    \label{prop}
    For the sequence \(\{\alpha_k\}\) generated by Algorithm~\ref{Algorithm: Accelerated Dinkelbach}, there exists no two 
    consecutive  iterates {\color{black}that} both lie strictly below $\alpha^{*}$.
    \end{proposition}
    
\begin{proof}
    {\color{black}
We begin by proving a key intermediate result.} For every index $k \geq 1$, following each update step of Algorithm~\ref{Algorithm: Accelerated Dinkelbach} (i.e., either $k \mapsto k+1$ or $k \mapsto k+2$), the newly generated iterate (corresponding to the updated index $k$) satisfies $\alpha_k \geq \alpha^*$. We verify this by analyzing all three iterate-generation procedures specified in the algorithm: \eqref{eq:alpha sequence}, \eqref{eq:min tangent}, and \eqref{eq:original Dinkelbach}.

 If $\alpha_{k+1}$ is generated by \eqref{eq:alpha sequence}, as specified by update rules, we perform two  checks: (1) If $g(\alpha_{k+1}) < 0$, then $\alpha_{k+2}$ is generated by \eqref{eq:min tangent}. 
 Analogously to the proof of Lemma~\ref{lemma:inequ_1}, we have $\alpha_{k+2} \geq \alpha^*$. After updating $k \leftarrow k+2$, the latest iterate $\alpha_k$ satisfies $\alpha_k \geq \alpha^*$.
       (2) If $g(\alpha_{k+1}) \ge 0$, we have $\alpha_{k+1} \ge  \alpha^*$. After updating $k \leftarrow k+1$, the latest iterate $\alpha_k$ satisfies $\alpha_k \geq \alpha^*$.
  If an $\alpha_{k+1}$ is generated by  \eqref{eq:original Dinkelbach}, similarly, we have $\alpha_{k+1}\ge \alpha^{*}$. After updating $k\leftarrow k+1$, the latest iterate $\alpha_{k}\ge \alpha^{*}$.
  We conclude that after updating $k$, the latest point satisfies $\alpha_{k}\ge \alpha^{*}$.

  Given  $\alpha_{-1},\alpha_{0}\ge \alpha^{*}$ 
  {\color{black}and every newly generated iterate also satisfies this bound, it is impossible for two consecutive iterates in the sequence to both lie strictly below
  $\alpha^{*}$.}
    \end{proof}
  
{\color{black}To further characterize the algorithm's behavior, we now examine the screening condition that triggers iteration \eqref{eq:alpha sequence}:
  \begin{equation}
  \label{eq:cond}
   g(\alpha_{k-1})f_{2}(x_{k})(\alpha_{k}-\alpha_{k-1})\le \rho g(\alpha_{k})(g(\alpha_{k})-g(\alpha_{k-1})),
  \end{equation}
where $\rho > 1$ is a predefined constant and $\alpha_k \geq \alpha^*$ by the intermediate result in Proposition~\ref{prop} (since $\alpha_k$ is the latest iterate).

Note that if $\alpha_{k-1} < \alpha^*$, then the left-hand side of \eqref{eq:cond} is negative while the right-hand side is positive, implying the inequality holds automatically. Furthermore, Theorem~\ref{theorem:accelerated convergence} will establish that the subsequence of $\{\alpha_k\}$ with non-negative $g$-values is strictly decreasing, which rules out the possibility of $\alpha^* < \alpha_{k-1} \leq \alpha_k$.

Consequently, \eqref{eq:cond} holds only in the following two cases:
\begin{itemize}
    \item \textbf{Case 1}: $\alpha_{k-1} < \alpha^* < \alpha_k$ (By Proposition~\ref{prop}, $\alpha_{k-1}<\alpha^{*}$ implies $ \alpha_{k}\ge \alpha^{*}$).
    \item \textbf{Case 2}: $\alpha^* < \alpha_k < \alpha_{k-1}$ satisfying~\eqref{eq:cond}.
\end{itemize}
We will analyze \eqref{eq:alpha sequence} under these two cases to establish the algorithm's convergence, and detailed discussions will be presented in the next subsection.}

\subsection{Global Convergence}
This subsection establishes the global convergence of accelerated Dinkelbach method~\ref{Algorithm: Accelerated Dinkelbach}. Our analysis proceeds as follows.
\begin{itemize}
    \item Lemma~\ref{Lemma:justice} ensures the well-posedness of the iteration \eqref{eq:alpha sequence} {\color{black}under either \textbf{Case 1} or \textbf{2}}, since it ensures that the denominator never vanishes.
    
    \item 
Lemmas~\ref{Lemma:case1} and~\ref{Lemma:case2} provide error bounds for iteration \eqref{eq:alpha sequence} in \textbf{Cases 1} and \textbf{2}, respectively.
    
    \item 
Theorem~\ref{theorem:accelerated convergence} establishes the global convergence of $\{\alpha_{k}\}$ to $\alpha^{*}$.
    
\item Corollary~\ref{coro:bound} deduces a uniform bound for $\{\alpha_{k}\}$.

\item Corollary~\ref{corollary} establishes the quadratic convergence rate for the subsequence of $\{\alpha_{k}\}$ with non-negative $g$-values.
\end{itemize}
\begin{lemma}\label{Lemma:justice}
Under \textbf{Case 1} or \textbf{Case 2}, the denominator in \eqref{eq:alpha sequence} is positive, ensuring that \eqref{eq:alpha sequence} is well-defined.
\end{lemma}
\begin{proof}
We first consider
\textbf{Case 1}   ($\alpha_{k-1} < \alpha^{*} < \alpha_{k}$).  
Since $g$ is strictly increasing, we have
\begin{equation}
\nonumber
g(\alpha_{k-1}) < g(\alpha^{*}) =0< g(\alpha_{k}).
\end{equation}
This implies that
\begin{equation}\label{eq:1(1-1)}
g(\alpha_{k})\big(g(\alpha_{k}) - g(\alpha_{k-1})\big) > 0.
\end{equation}
Since $f_2(x_k) > 0$ and $\alpha_k > \alpha_{k-1}$, we obtain  
\begin{equation}\label{eq:-11(1-1)}
-g(\alpha_{k-1})f_2(x_k)(\alpha_k - \alpha_{k-1}) > 0.
\end{equation}
These inequalities  \eqref{eq:1(1-1)}  and \eqref{eq:-11(1-1)} collectively guarantee that the denominator in \eqref{eq:alpha sequence} is positive.

Now consider \textbf{Case 2} ($\alpha^{*} < \alpha_{k} < \alpha_{k-1}$ satisfying~\eqref{eq:cond}).
 Dividing both sides of \eqref{eq:cond} by $\rho\,( > 1)$ yields  
\[
\frac{1}{\rho}g(\alpha_{k-1})f_2(x_k)(\alpha_k - \alpha_{k-1}) \leq g(\alpha_{k})(g(\alpha_{k}) - g(\alpha_{k-1})).
\]
We thus have a lower bound for the denominator in~\eqref{eq:alpha sequence}:
\begin{equation}\label{eq:case 2 denominator}
g(\alpha_{k})(g(\alpha_{k}) - g(\alpha_{k-1})) - g(\alpha_{k-1})f_2(x_k)(\alpha_k - \alpha_{k-1}) \geq  \frac{1-\rho}{\rho}g(\alpha_{k-1})f_2(x_k)(\alpha_k - \alpha_{k-1}).
\end{equation}
The facts
$\rho>1$, $ g(\alpha_{k-1})>g(\alpha^{*})=0$,
$f_2(x_k) > 0$ and $\alpha_k - \alpha_{k-1} < 0$ ensure that the denominator remains positive.  
Our proof is complete.
\end{proof}

To establish the convergence of the accelerated Dinkelbach method, we first reformulate \eqref{eq:alpha sequence} as 
\begin{align}
         &\alpha_{k+1}-\alpha^{*}\nonumber\\=&\alpha_{k}-\alpha^{*}-\frac{g(\alpha_{k})(g(\alpha_{k})-g(\alpha_{k-1}))(\alpha_{k}-\alpha_{k-1})}{g(\alpha_{k})(g(\alpha_{k})-g(\alpha_{k-1}))-g(\alpha_{k-1})f_{2}(x_k)(\alpha_{k}-\alpha_{k-1})}\nonumber\\
    =&\frac{g(\alpha_{k})(g(\alpha_{k})-g(\alpha_{k-1}))(\alpha_{k-1}-\alpha^{*})-g(\alpha_{k-1})f_{2}(x_k)(\alpha_{k}-\alpha_{k-1})(\alpha_{k}-\alpha^{*})}{g(\alpha_{k})(g(\alpha_{k})-g(\alpha_{k-1}))-g(\alpha_{k-1})f_{2}(x_k)(\alpha_{k}-\alpha_{k-1})}\nonumber\\
    =&(\alpha_{k}-\alpha^{*})
\frac{\delta g(\alpha_{k-1},\alpha^{*})f_{2}(x_{k})-\delta g(\alpha_{k},\alpha^{*})\delta g (\alpha_{k},\alpha_{k-1})}{\delta g(\alpha_{k-1},\alpha^{*})f_{2}(x_{k})-\delta g(\alpha_{k},\alpha^{*})\delta g (\alpha_{k},\alpha_{k-1})\frac{\alpha_{k}-\alpha^{*}}{\alpha_{k-1}-\alpha^{*}} },\label{eq:simpler}
\end{align}
where $g(\alpha^{*})=0$ is used in the last equality.
{\color{black}
The error satisfies the bounds in the following lemmas.}
\begin{lemma}\label{Lemma:case1}
    For $k=0,1,\cdots$, suppose $\alpha_{k+1}$ is generated by \eqref{eq:alpha sequence} and  the previous points
     $\alpha_{k}$ and $\alpha_{k-1}$ satisfying \textbf{Case 1}  ($\alpha_{k-1} < \alpha^{*} < \alpha_{k}$), the behavior of $\alpha_{k+1}$ splits into
\begin{itemize}
    \item \textbf{Case 1.1:} $\alpha_{k+1} \geq \alpha^{*}$. This subcase leads to
    $$
    \alpha_{k+1} - \alpha^{*} \leq \alpha_{k} - \alpha^{*} - \frac{g(\alpha_{k})}{f_{2}(x_{k})}.
    $$
    
    \item \textbf{Case 1.2:} $\alpha_{k+1} < \alpha^{*}$. This subcase leads to
    $$
    |\alpha_{k+1} - \alpha^{*}| \leq M |\alpha_{k} - \alpha^{*}|,
    $$
    with 
    \begin{equation}\label{eq:M}
    M = \frac{f_{2}(x_{-1}) - \min\limits_{\xi \in \mathcal{F}} f_{2}(\xi)}{\min\limits_{\xi \in \mathcal{F}} f_{2}(\xi)} > 0.
    \end{equation}
\end{itemize}
\end{lemma}
\begin{proof}
 Since $\alpha_{k-1}<\alpha^{*}<\alpha_{k}$ and $\delta g>0$  by Corollary~\ref{coro_supply}, we have a lower estimate of the denominator:
\begin{equation}\label{eq:lower denominator bound}
\delta g(\alpha_{k-1},\alpha^{*})f_{2}(x_{k}) - \delta g(\alpha_{k},\alpha^{*})\delta g(\alpha_{k},\alpha_{k-1})\frac{\alpha_{k}-\alpha^{*}}{\alpha_{k-1}-\alpha^{*}} \geq \delta g(\alpha_{k-1},\alpha^{*})f_{2}(x_{k}).
\end{equation}
By Lemma~\ref{lemma:subgradient_property}, we get the key inequalities from $\alpha_{k-1}\le\alpha^{*}<\alpha_{k}$:
\begin{equation}
\label{eq:five}
f_{2}(x_{k-1})\le \delta g(\alpha_{k-1},\alpha^{*})\le f_{2}(x^{*})\le \delta g(\alpha^{*},\alpha_{k})\le f_{2}(x_{k}),\, \delta g(\alpha_{k},\alpha_{k-1})\le f_{2}(x_{k}). 
\end{equation}
Besides, given that $\delta g(\alpha^{*},\alpha_{k-1})\le \delta g(\alpha_{k},\alpha^{*})$, 
we derive the following inequality by Sugar Water Inequality:
\begin{equation}
\label{eq:sugar water}
\delta g(\alpha_{k-1},\alpha^{*})\le \frac{g(\alpha^{*})-g(\alpha_{k-1})+g(\alpha_{k})-g(\alpha^{*})}{\alpha^{*}-\alpha_{k-1}+\alpha_{k}-\alpha^{*}}=\delta g(\alpha_{k-1},\alpha_{k}).
\end{equation}
If $\alpha_{k+1}$ falls into \textbf{Case 1.1} (i.e., $\alpha_{k+1}
\ge \alpha^{*}$), the numerator of \eqref{eq:simpler} is positive. \eqref{eq:sugar water} leads to a upper estimate of the numerator
\begin{equation}\label{eq:numerator upper bound}
\delta g(\alpha_{k-1},\alpha^{*})f_{2}(x_{k})-\delta g(\alpha_{k},\alpha^{*})\delta g(\alpha_{k},\alpha_{k-1})\leq \delta g(\alpha_{k-1},\alpha^{*})(f_{2}(x_{k})-\delta g(\alpha_{k},\alpha^{*})).
\end{equation}
Thus we derive the following bound:
\begin{align}
    \alpha_{k+1}-\alpha^{*}&= (\alpha_{k}-\alpha^{*})
\frac{\delta g(\alpha_{k-1},\alpha^{*})f_{2}(x_{k})-\delta g(\alpha_{k},\alpha^{*})\delta g (\alpha_{k},\alpha_{k-1})}{\delta g(\alpha_{k-1},\alpha^{*})f_{2}(x_{k})-\delta g(\alpha_{k},\alpha^{*})\delta g (\alpha_{k},\alpha_{k-1})\frac{\alpha_{k}-\alpha^{*}}{\alpha_{k-1}-\alpha^{*}}}\nonumber\\
&\le (\alpha_{k}-\alpha^{*})\frac{\delta g(\alpha_{k-1},\alpha^{*})(f_{2}(x_{k})-\delta g(\alpha_{k},\alpha^{*}))}{\delta g(\alpha_{k-1},\alpha^{*})f_{2}(x_{k})} \ (\text{by}\ \eqref{eq:lower denominator bound} \ \text{and}\ \eqref{eq:numerator upper bound})\nonumber\\
&=(\alpha_{k}-\alpha^{*})\frac{f_{2}(x_{k})-\delta g(\alpha_{k},\alpha^{*})}{f_{2}(x_{k})} \nonumber\\
&=\alpha_{k}-\alpha^{*}-\frac{g(\alpha_{k})}{f_{2}(x_{k})}.\nonumber
\end{align}
If $\alpha_{k+1}$ falls into \textbf{Case 1.2} (i.e., $\alpha_{k+1}
< \alpha^{*}$), the numerator of \eqref{eq:simpler} is negative. Similar argument indicates
\begin{align}
\alpha^{*}-\alpha_{k+1}&= (\alpha_{k}-\alpha^{*})
\frac{\delta g(\alpha_{k},\alpha^{*})\delta g (\alpha_{k},\alpha_{k-1})-\delta g(\alpha_{k-1},\alpha^{*})f_{2}(x_{k})}{\delta g(\alpha_{k-1},\alpha^{*})f_{2}(x_{k})-\delta g(\alpha_{k},\alpha^{*})\delta g (\alpha_{k},\alpha_{k-1})\frac{\alpha_{k}-\alpha^{*}}{\alpha_{k-1}-\alpha^{*}}}\nonumber\\
&\le (\alpha_{k}-\alpha^{*})\frac{\delta g(\alpha_{k},\alpha^{*})\delta g(\alpha_{k},\alpha_{k-1})-\delta g(\alpha_{k-1},\alpha^{*})f_{2}(x_{k})}{\delta g(\alpha_{k-1},\alpha^{*})f_{2}(x_{k})} \,( \text{by}\ \eqref{eq:lower denominator bound})\nonumber\\ 
&\le (\alpha_{k}-\alpha^{*})\frac{f_{2}(x_{k})f_{2}(x_{k})-f_{2}(x_{k-1})f_{2}(x_{k})}{f_{2}(x_{k-1})f_{2}(x_{k})}\,( \text{by}\ \eqref{eq:five})\nonumber\\
&= (\alpha_{k}-\alpha^{*})\frac{f_{2}(x_{k})-f_{2}(x_{k-1})}{f_{2}(x_{k-1})}\nonumber\\
&\le M(\alpha_{k}-\alpha^{*})\label{eq:boundeded},
\end{align}
\noindent where \eqref{eq:boundeded} follows from the definition of $M$ in \eqref{eq:M}.
The proof is complete.
\end{proof}
Figure~\ref{fig:total2} provides a graphical 
representation of Lemma~\ref{Lemma:case1}.
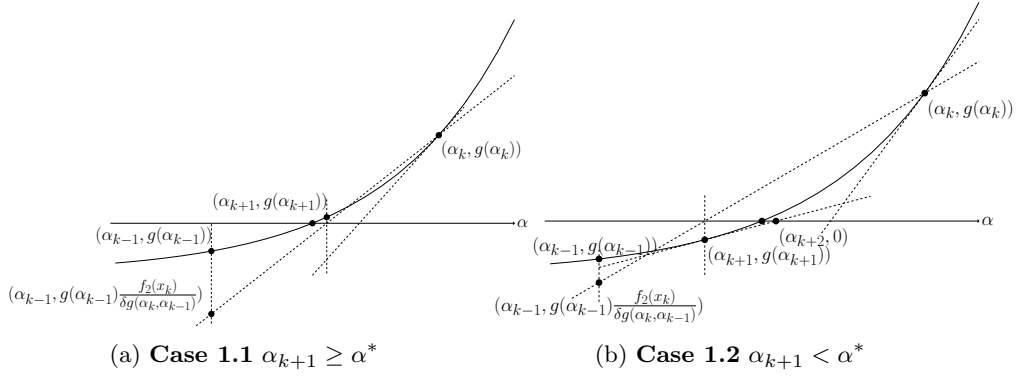
\begin{figure}[h]
        \hfill 
    \begin{subfigure}[b]{0.48\textwidth} 
        \centering
        \begin{adjustbox}{width=1.1\linewidth} 
        \begin{tikzpicture}[scale=2.5, font=\Huge] 
            
            \draw[->] (-4,0) -- (4,0)node[right, font=\Huge] {$\alpha$};
            \draw[domain=-3.9:4]plot(\x,{1.5^(\x)-1});
            \draw[fill=black] (0,0)circle (1.5pt); 
            \draw[fill=black] (2.5,1.75567596) circle (1.5pt)node[below right, font=\Huge, yshift=-5pt] {$(\alpha_{k},g(\alpha_{k}))$}; 
            
            \draw[fill=black](-2,-0.5556) circle (1.5pt)node[above left, font=\Huge, xshift=5pt, yshift=5pt] {$(\alpha_{k-1},g(\alpha_{k-1}))$};
            \draw[fill=black](-2,-1.80832764) circle (1.5pt)node[above left, font=\Huge, xshift=-8pt, yshift=0pt] {$(\alpha_{k-1},g(\alpha_{k-1})\frac{f_{2}(x_{k})}{\delta g (\alpha_{k},\alpha_{k-1})})$};
            \draw[dashed, line width=0.5][domain=0:3]plot(\x,{1.11733045099*(\x-2.5)+1.75567596});
            \draw[dashed, line width=0.5][domain=-2:0]plot(-2,\x);
            \draw[dashed, line width=0.5][domain=-2.3:4]plot(\x,{0.7920008*(\x-2.5)+1.75567596});
            \draw[dashed, line width=0.5][domain=-1:0.5]plot(0.283239739,\x);

            \draw[fill=black] (0.283239739,0.121698249841035) circle (1.5pt) node[above left, font=\Huge, xshift=5pt, yshift=5pt] {$(\alpha_{k+1},g(\alpha_{k+1}))$}; 
        \end{tikzpicture}
        \end{adjustbox}
        \caption{\textbf{Case 1.1} $\alpha_{k+1}\ge \alpha^{*}$}
        \label{fig_tikz2}
    \end{subfigure}
    \centering
    \begin{subfigure}[b]{0.48\textwidth} 
        \centering
        \begin{adjustbox}{width=1.1\linewidth} 
        \begin{tikzpicture}[scale=2.5, font=\Huge] 
            \draw[->] (-4,0) -- (4,0) node[anchor=west, font=\Huge]  {$\alpha$};
            \draw[domain=-3.9:4]plot(\x,{1.5^(\x)-1});
            \draw[fill=black] (0,0)circle (1.5pt);
            \draw[fill=black] (3,2.375) circle (1.5pt)node[below right, font=\Huge, yshift=-5pt] {$(\alpha_{k},g(\alpha_{k}))$}; 

            \draw[dashed, line width=0.5][domain=-3.5:4]plot(\x,{(0.475+0.11111)*(\x-3)+2.375});
            
            \draw[dashed, line width=0.5][domain=-1:0.5]plot(-1.05213278,\x);
            
            \draw[dashed, line width=0.5][domain=-1.5:-0.5]plot(-3,\x);\draw[fill=black](-3,-0.703703703704) circle (1.5pt)node[above, font=\Huge, xshift=-5pt] {$(\alpha_{k-1},g(\alpha_{k-1}))$}; \draw[fill=black] (-3,-1.14166) circle (1.5pt) node[ below,font=\Huge, yshift=-5pt] {$(\alpha_{k-1},g(\alpha_{k-1})\frac{f_{2}(x_{k})}{\delta g(\alpha_{k},\alpha_{k-1})})$};
            \draw[fill=black] (-1.05213278,-0.347277454) circle (1.5pt) node[below right, font=\Huge, yshift=-5pt] {$(\alpha_{k+1},g(\alpha_{k+1}))$}; 
            \draw[dashed, line width=0.5][domain=1.1:4]plot(\x,{1.368444739*(\x-3)+2.375});
            
            \draw[dashed, line width=0.5][domain=-3:2]plot(\x,{0.264656147*(\x+1.05213278)-0.347277454});
            \draw[fill=black] (0.26,0) circle (1.5pt)node[below right, font=\Huge, yshift=-5pt] {$(\alpha_{k+2},0)$}; 
        \end{tikzpicture}
        \end{adjustbox}
\caption{\textbf{Case 1.2} $\alpha_{k+1}< \alpha^{*}$}
        \end{subfigure}
    \caption{In
    \textbf{Case 1.1},  
   the proposed method demonstrates better approximation  
than the Dinkelbach method.
    In \textbf{Case 1.2}, the approximation error is bounded by $M|\alpha_k - \alpha^*|$. The subsequent iterate $\alpha_{k+2}$ is determined by the tangents at $\alpha_{k}$ and $\alpha_{k-1}$.}  
    \label{fig:total2}
\end{figure}
\begin{lemma}\label{Lemma:case2}
    For $k=0,1,\cdots$, suppose $\alpha_{k+1}$ is generated by \eqref{eq:alpha sequence} and  the previous points
     $\alpha_{k}$ and $\alpha_{k-1}$ satisfying \textbf{Case 2}  ($\alpha^* < \alpha_k < \alpha_{k-1}$ under \eqref{eq:cond}), the behavior of $\alpha_{k+1}$ splits into
\begin{itemize}
    \item \textbf{Case 2.1}: $\alpha_{k+1} \geq \alpha^{*}$. This subcase leads to
    $$
    \alpha_{k+1} - \alpha^{*} \leq \alpha_{k} - \alpha^{*} - \frac{g(\alpha_{k})}{f_{2}(x_{k})}.
    $$
    
    \item \textbf{Case 2.2}: $\alpha_{k+1} < \alpha^{*}$. This subcase leads to
    $$
    |\alpha_{k+1} - \alpha^{*}| \leq M' |\alpha_{k} - \alpha^{*}|,
    $$
 where 
    \begin{equation}
    \label{eq:M'}
    M' = \frac{\rho \big( f_{2}(x_{-1}) - f_{2}(x^{*}) \big)}{(\rho-1)f_{2}(x^{*})} > 0.
    \end{equation}
\end{itemize}
\end{lemma}
\begin{proof}
In \textbf{Case 2}, the denominator of \eqref{eq:simpler} is positive by Lemma~\ref{Lemma:justice}.
By Lemma~\ref{lemma:subgradient_property}, we get the following key inequalities  from $\alpha^{*}<\alpha_{k}<\alpha_{k-1}$:
\begin{equation}
\label{eq:five1}
    f_{2}(x^{*})\le \delta g(\alpha^{*},\alpha_{k})\le f_{2}(x_{k})\le \delta g(\alpha_{k},\alpha_{k-1})\le f_{2}(x_{k-1}),\, \delta g(\alpha_{k-1},\alpha^{*})\ge f_{2}(x^{*}).
\end{equation}
If $\alpha_{k+1}$ falls into \textbf{Case 2.1} (i.e., $\alpha_{k+1}\ge \alpha^{*}$), the numerator of \eqref{eq:simpler} is positive. 
Then we derive the upper estimate.
\begin{align}
    \alpha_{k+1}-\alpha^{*}&=(\alpha_{k}-\alpha^{*})\left(1 - \frac{\delta g(\alpha_{k},\alpha^{*})\delta g(\alpha_{k},\alpha_{k-1})\frac{\alpha_{k-1}-\alpha_{k}}{\alpha_{k-1}-\alpha^{*}}}{\delta g(\alpha_{k-1},\alpha^{*})f_{2}(x_{k}) - \delta g(\alpha_{k},\alpha^{*})\delta g (\alpha_{k},\alpha_{k-1})\frac{\alpha_{k}-\alpha^{*}}{\alpha_{k-1}-\alpha^{*}}}\right)\nonumber\\
    &\le(\alpha_{k}-\alpha^{*})\left(1 - \frac{\delta g(\alpha_{k},\alpha^{*})\delta g(\alpha_{k},\alpha_{k-1})\frac{\alpha_{k-1}-\alpha_{k}}{\alpha_{k-1}-\alpha^{*}}}{\delta g(\alpha_{k-1},\alpha^{*})f_{2}(x_{k})-\delta g(\alpha_{k},\alpha^{*})f_{2}(x_{k})\frac{\alpha_{k}-\alpha^{*}}{\alpha_{k-1}-\alpha^{*}}}\right)\ (\text{by}\ \eqref{eq:five1})\nonumber\\
    &=(\alpha_{k}-\alpha^{*})\left(1 - \frac{\delta g(\alpha_{k},\alpha^{*})\delta g(\alpha_{k},\alpha_{k-1})\frac{\alpha_{k-1}-\alpha_{k}}{\alpha_{k-1}-\alpha^{*}}}{f_{2}(x_{k})\left(\delta g(\alpha_{k-1},\alpha^{*})-\delta g(\alpha_{k},\alpha^{*})\frac{\alpha_{k}-\alpha^{*}}{\alpha_{k-1}-\alpha^{*}}\right)}\right)\nonumber\\
    &=(\alpha_{k}-\alpha^{*})\left(1 - \frac{\delta g(\alpha_{k},\alpha^{*})\delta g(\alpha_{k},\alpha_{k-1})\frac{\alpha_{k-1}-\alpha_{k}}{\alpha_{k-1}-\alpha^{*}}}{f_{2}(x_{k})\delta g(\alpha_{k-1},\alpha_{k})\frac{\alpha_{k-1}-\alpha_{k}}{\alpha_{k-1}-\alpha^{*}}}\right)\nonumber\\
    &=(\alpha_{k}-\alpha^{*})\left(1 - \frac{\delta g(\alpha_{k},\alpha^{*})}{f_{2}(x_{k})}\right)=\alpha_{k}-\alpha^{*}-\frac{g(\alpha_{k})}{f_{2}(x_{k})}.\nonumber
\end{align}
If $\alpha_{k+1}$ falls into \textbf{Case 2.2} (i.e., $\alpha_{k+1}< \alpha^{*}$), the numerator of \eqref{eq:simpler} is negative. 
We have the upper estimate.
\begin{align}
\alpha^{*}-\alpha_{k+1}&\le (\alpha_{k}-\alpha^{*})\frac{\delta g(\alpha_{k},\alpha^{*})\delta g(\alpha_{k},\alpha_{k-1})-\delta g(\alpha_{k-1},\alpha^{*})f_{2}(x_{k})}{\left(1-\frac{1}{\rho}\right)\delta g(\alpha_{k-1},\alpha^{*})f_{2}(x_{k})}\ (\text{by}\, \eqref{eq:case 2 denominator})\nonumber\\
&\le (\alpha_{k}-\alpha^{*})\frac{f_{2}(x_{k})f_{2}(x_{k-1})-f_{2}(x^{*})f_{2}(x_{k})}{\left(1-\frac{1}{\rho}\right)f_{2}(x^{*})f_{2}(x_{k})} \ (\text{by} \ \eqref{eq:five1})\nonumber\\
&\le (\alpha_{k}-\alpha^{*})\frac{f_{2}(x_{k-1})-f_{2}(x^{*})}{\left(1-\frac{1}{\rho}\right)f_{2}(x^{*})}\nonumber\\
&\le M'(\alpha_{k}-\alpha^{*}),\label{eq:detail1}
\end{align}
where \eqref{eq:detail1} follows from the definition of $M'$ in \eqref{eq:M'}. The proof is complete.
\end{proof}
{\color{black}
Figure~\ref{fig:total} provides a graphical 
representation of Lemma~\ref{Lemma:case2}.}
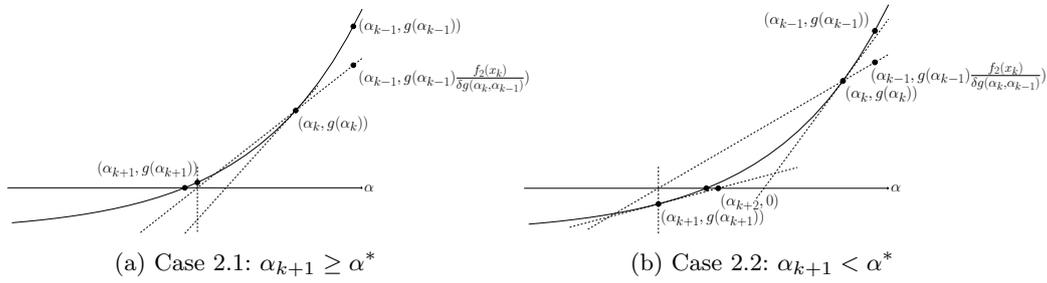
\begin{figure}[htbp]
    \begin{subfigure}[b]{0.48\textwidth} 
        \centering
        \begin{adjustbox}{width=1.1\linewidth} 
        \begin{tikzpicture}[scale=2.5, font=\Huge] 
            
            \draw[->] (-4,0) -- (4,0) node[anchor=west, font=\Huge]  {$\alpha$};
            \draw[domain=-3.9:4]plot(\x,{1.5^(\x)-1});
            \draw[fill=black] (0,0)circle (1.5pt); 
            \draw[fill=black] (2.5,1.75567596) circle (1.5pt)node[below right, font=\Huge, yshift=-5pt] {$(\alpha_{k},g(\alpha_{k}))$}; 
            \draw[fill=black] (3.8,3.66817130) circle (1.5pt)node[right, font=\Huge, xshift=5pt] {$(\alpha_{k-1},g(\alpha_{k-1}))$}; 
            \draw[fill=black] (3.8,2.785277) circle (1.5pt)node[below right, font=\Huge, xshift=5pt, yshift=5pt] {$(\alpha_{k-1},g(\alpha_{k-1})\frac{f_{2}(x_{k})}{\delta g (\alpha_{k},\alpha_{k-1})})$}; 
            \draw[dashed, line width=0.5][domain=0:3]plot(\x,{1.11733045099*(\x-2.5)+1.75567596});
            \draw[dashed, line width=0.5][domain=-1:4]plot(\x,{0.7920008*(\x-2.5)+1.75567596});
            \draw[dashed, line width=0.5][domain=-1:0.5]plot(0.283239739,\x);

            \draw[fill=black] (0.283239739,0.121698249841035) circle (1.5pt) node[above left, font=\Huge, xshift=5pt, yshift=5pt] {$(\alpha_{k+1},g(\alpha_{k+1}))$}; 
        \end{tikzpicture}
        \end{adjustbox}
        \caption{Case 2.1: $\alpha_{k+1}\ge \alpha^{*}$}
        \label{fig:tikz2}
    \end{subfigure}
        \centering
        \hfill 
    \begin{subfigure}[b]{0.48\textwidth} 
        \centering
        \begin{adjustbox}{width=1.1\linewidth} 
        \begin{tikzpicture}[scale=2.5, font=\Huge] 
            \draw[->] (-4,0) -- (4,0) node[anchor=west, font=\Huge]  {$\alpha$};
            \draw[domain=-3.9:4]plot(\x,{1.5^(\x)-1});
            \draw[fill=black] (0,0)circle (1.5pt);
            \draw[fill=black] (3,2.375) circle (1.5pt)node[below right, font=\Huge, yshift=-5pt] {$(\alpha_{k},g(\alpha_{k}))$}; 
            \draw[fill=black] (3.7,3.482679) circle (1.5pt)node[above left, font=\Huge, xshift=-5pt] {$(\alpha_{k-1},g(\alpha_{k-1}))$}; 
            \draw[fill=black] (3.7,2.785277) circle (1.5pt)node[below right, font=\Huge, xshift=-10pt, yshift=5pt] {$(\alpha_{k-1},g(\alpha_{k-1})\frac{f_{2}(x_{k})}{\delta g (\alpha_{k},\alpha_{k-1})})$}; 

            \draw[dashed, line width=0.5][domain=-2.6:4]plot(\x,{(0.475+0.11111)*(\x-3)+2.375});
            \draw[dashed, line width=0.5][domain=-1:0.5]plot(-1.05213278,\x);

            \draw[fill=black] (-1.05213278,-0.347277454) circle (1.5pt) node[below right, font=\Huge, yshift=-5pt] {$(\alpha_{k+1},g(\alpha_{k+1}))$}; 
            \draw[dashed, line width=0.5][domain=1.1:4]plot(\x,{1.368444739*(\x-3)+2.375});
            \draw[dashed, line width=0.5][domain=-3:2]plot(\x,{0.264656147*(\x+1.05213278)-0.347277454});
            \draw[fill=black] (0.26,0) circle (1.5pt)node[below right, font=\Huge, yshift=-5pt] {$(\alpha_{k+2},0)$}; 
        \end{tikzpicture}
        \end{adjustbox}
\caption{Case 2.2: $\alpha_{k+1}< \alpha^{*}$}

        \label{fig:tikz1}
    \end{subfigure}
    \caption{
In \textbf{Case 2.1}, the iteration~\eqref{eq:alpha sequence} demonstrates superior approximation compared to the Dinkelbach method. 
In \textbf{Case 2.2}, the approximation error is bounded by $M'|\alpha_k - \alpha^*|$.
The subsequent iterate $\alpha_{k+2}$ is determined by the tangents at $\alpha_{k}$ and $\alpha_{k-1}$.} 
    \label{fig:total}
\end{figure}

We now establish the global convergence of the accelerated Dinkelbach algorithm.
\begin{theorem}\label{theorem:accelerated convergence}
    For any initial $x_{-1} \in \mathcal{F}$, the sequence $\{\alpha_k\}$ generated by the accelerated Dinkelbach algorithm converges globally to $\alpha^*$.
\end{theorem}
\begin{proof}
We assume $\{\alpha_{k}\}$ is an infinite sequence. 
{\color{black}
We partition $\{\alpha_{k}\}$ into two disjoint subsequences based on their position relative to $\alpha^{*}$:}
\begin{itemize}
\item $\{\alpha_{p}^{+}\}=\{\alpha_{k}:\alpha_{k}\ge \alpha^{*}\}=\{\alpha_{k}:g(\alpha_{k})\ge 0\}$,
\item $\{\alpha_{q}^{-}\}=\{\alpha_{k}:\alpha_{k}< \alpha^{*}\}=\{\alpha_{k}:g(\alpha_{k})< 0\}$,
\end{itemize}
{\color{black}
 where $p,q=1,2,\cdots$ (indexing starts from 1).
Proposition~\ref{prop} guarantees that $\{\alpha_p^+\}$ is infinite.}
For the sequence $\{\alpha_{p}^{+}\}$, we prove by mathematical induction that
\begin{equation}\label{eq:form Dinkelbach}
\alpha_{p}^{+} \leq \alpha_{p-1}^{+} - \frac{g(\alpha_{p-1}^{+})}{f_{2}(x_{p-1}^{+})}, \quad p = 2,3,\cdots,
\end{equation}
where $x_{p}^{+} \in \argmax_{x \in \mathcal{F}} \left\{ -f_{1}(x) + \alpha_{p}^{+} f_{2}(x) \right\}$.
Note $\alpha_{1}^{+} = \alpha_{-1} \geq \alpha^{*}$, as $\alpha_{-1} = f_1(x_{-1})/f_2(x_{-1}) \geq \min_{x \in \mathcal{F}} f_1(x)/f_2(x)$. 
From the definition of $\alpha_0 = \alpha_{-1} - g(\alpha_{-1})/f_{2}(x_{-1})$, we have
\(
\alpha_{2}^{+} = \alpha_0\) and $\alpha_{2}^{+}$ satisfies \eqref{eq:form Dinkelbach} with equality.
For the induction step ($m \geq 3$), assume \eqref{eq:form Dinkelbach} holds for all $p \leq m-1$. 
Given $\alpha_{m'}=\alpha_m^+ $ for some $m'$, we analyze two cases below.
\begin{itemize}
    \item If $\alpha_{m'-1} < \alpha^{*}$, according to Proposition~\ref{prop}, it leads to $\alpha^{*} \leq \alpha_{m'-2} = \alpha_{m-1}^{+}$. From \eqref{eq:min tangent}, we derive
    \[
    \alpha_{m}^{+} = \alpha_{m'} = \min\left\{\alpha_{m'-1} - \frac{g(\alpha_{m'-1})}{f_{2}(x_{m'-1})}, \alpha_{m'-2} - \frac{g(\alpha_{m'-2})}{f_{2}(x_{m'-2})}\right\} \leq \alpha_{m-1}^{+} - \frac{g(\alpha_{m-1}^{+})}{f_{2}(x_{m-1}^{+})}.
    \]
\item If $\alpha_{m'-1}\ge \alpha^{*}$ (i.e., $\alpha_{m-1}^{+}=\alpha_{m'-1}$), we further examine $\alpha_{m'-2}$. If $\alpha_{m'-2}< \alpha^{*}$, 
then it holds $\alpha_{m'-2}<\alpha^{*}<\alpha_{m'-1}$. 
Our update rule yields $\alpha_{m'}$ generated by \eqref{eq:alpha sequence} corresponding to \textbf{Case 1.1}. By Lemma~\ref{Lemma:case1}, we immediately get
\[
\alpha_{m}^{+}=\alpha_{m'}\le \alpha_{m'-1}-\frac{g(\alpha_{m'-1})}{f_{2}(x_{m'-1})}=\alpha_{m-1}^+-\frac{g(\alpha_{m-1}^+)}{f_{2}(x_{m-1}^{+})}.
\]
If $\alpha_{m'-2}\ge \alpha^{*}$, then $\alpha_{m-2}^{+}=\alpha_{m'-2}$.
 Hypothesis ensures 
 \[
\alpha^{*}\le\alpha_{m'-1}=\alpha_{m-1}^{+}\le \alpha_{m-2}^{+}-\frac{g(\alpha_{m-2}^{+})}{f_{2}(x_{m-2}^{+})}\le \alpha_{m-2}^{+}=\alpha_{m'-2}.
 \]\noindent According to our update rule, if $\alpha_{m-2}^{+}$ and $\alpha_{m-1}^{+}$ satisfy \eqref{eq:cond}, then $\alpha_{m'}$ generates by \eqref{eq:alpha sequence} corresponding to \textbf{Case 2.1}. By Lemma~\ref{Lemma:case2}, we get
\[
\alpha_{m}^{+}=\alpha_{m'}\le \alpha_{m'-1}-\frac{g(\alpha_{m'-1})}{f_{2}(x_{m'-1})}=\alpha_{m-1}^+-\frac{g(\alpha_{m-1}^+)}{f_{2}(x_{m-1}^{+})}.
\]
If \eqref{eq:cond} fails for $\alpha_{m-1}^{+}$ and $\alpha_{m-2}^{+}$, then our update rule gives $\alpha_{m}^{+}$ by \eqref{eq:original Dinkelbach}, i.e.,
\[
\alpha_{m}^{+}=\alpha_{m-1}^{+}-\frac{g(\alpha_{m-1}^{+})}{f_{2}(x_{m-1}^{+})}.
\] 
\end{itemize}
{\color{black}
 Thus, the proof of \eqref{eq:form Dinkelbach} is complete.
Moreover, 
 analogous to the convergence in Theorem~\ref{theorem: converge},
\eqref{eq:form Dinkelbach} yields the convergence of $\{\alpha_{p}^{+}\}$ to $\alpha^{*}$. }

Finally, we establish the convergence of $\{\alpha_{k}\}$ to $\alpha^{*}$ through the following cases.
\begin{itemize}
    \item $\{\alpha_{q}^{-}\}$ is finite. There exists an integer $N_{0}$ such that for all $k \geq N_{0}+1$,  $\alpha^{*} \leq \alpha_{k} \in \{\alpha_{p}^{+}\}$. The convergence of $\{\alpha_{k}\}$ then follows directly from that of $\{\alpha_{p}^{+}\}$.
    \item $\{\alpha_{q}^{-}\}$ is infinite. 
    Since $\{\alpha_{p}^{+}\}$ is decreasing and  convergent, for any $\epsilon>0$, we get an index $p_{0}$, such that for all $p\ge p_{0}$, it holds
\begin{equation}
\label{eq:minim}
    |\alpha_{p}^{+}-\alpha^{*}|< \min\left\{\frac{\epsilon}{\max\{M,M'\}},\epsilon\right\},
\end{equation}
    where $M$ and $M'$ are defined in \eqref{eq:M} and \eqref{eq:M'},  respectively.
    Find the index $k_{0}$ such that $\alpha_{k_{0}}=\alpha_{p_{0}}^{+}$.
    Then for all $k\ge k_{0}$, we consider the following two cases.
If $\alpha_{k}\ge \alpha^{*}$, by the monotonicity of $\{\alpha_{p}^{+}\}$ and \eqref{eq:minim}, we have
        \[
        |\alpha_{k}-\alpha^{*}|\le |\alpha_{p_{0}}^{+}-\alpha^{*}|< \min\left\{\frac{\epsilon}{\max\{M,M'\}},\epsilon\right\}\le \epsilon.
        \]     
If $\alpha_k < \alpha^*$, the update rules for $k-1 \geq k_0$ imply $\alpha^* \leq \alpha_{k-1} \leq \alpha_{k_0}$.
We obtain  a chain of inequalities:
        \begin{align}
            |\alpha_{k}-\alpha^{*}|&\le \max\{M,M'\}|\alpha_{k-1}-\alpha^{*}|\ (\text{by Lemmas~\ref{Lemma:case1} and \ref{Lemma:case2}})\nonumber\\
            &\le \max\{M,M'\}|\alpha_{k_{0}}-\alpha^{*}|\ (\text{from} \ \alpha^{*}\le\alpha_{k-1}\le \alpha_{k_{0}})\nonumber\\
            &< \max\{M,M'\}\min\left\{\frac{\epsilon}{\max\{M,M'\}},\epsilon\right\}\ (\text{using \eqref{eq:minim}})\nonumber\\
            & \le \epsilon.\nonumber     \end{align}
\end{itemize}
{\color{black}We complete the proof.}
\end{proof}

{\color{black}
Below, we present two  corollaries and a key assumption, which play an essential role in analyzing the convergence of the sequence $\{\alpha_k\}$ in the subsequent subsection.}

\begin{corollary}
\label{coro:bound}
For any initial point $x_{-1}\in \mathcal{F}$, the sequence $\{\alpha_{k}\}$ admits a uniform bound:
\begin{equation}
\alpha_{-1}\ge \alpha_{k}\ge \alpha_{\text{low}}:=\alpha^{*}-\max\{M,M'\}(\alpha_{0}-\alpha^{*}),\quad   k=-1,0,\cdots,
\end{equation}
where $M$ and $M'$ are defined in \eqref{eq:M} and \eqref{eq:M'}, respectively.
\end{corollary}
\begin{proof}
We analyze two cases based on the relation between $\alpha_k$ and $\alpha^*$.
\begin{itemize}
    \item For any index $k$ satisfying $\alpha_{k}\ge \alpha^{*}$, by the definition of $\{\alpha_p^+\}$, it follows that  $\alpha_k \ge \alpha^*$. Since $\{\alpha_p^+\}$ is a decreasing sequence, $\alpha_{-1} = \alpha_1^+$ (the initial term of $\{\alpha_p^+\}$), and the initial term of a decreasing sequence is its maximum, we can obtain:
    \begin{equation}
\nonumber
\alpha_{-1} = \alpha_1^+ \ge \alpha_k \ge \alpha^*.
\end{equation}
Moreover, given $M,M'\ge 0$ and $\alpha_{0}\ge \alpha^{*}$, we have
\begin{equation}
\label{eq: positive bound}
\alpha_{-1}\ge \alpha^{*}\ge \alpha_{k}\ge\alpha^{*}-\max\{M,M'\}(\alpha_{0}-\alpha^{*})=\alpha_{\text{low}}.
\end{equation}
\item For any $k\ge 1$ with  $\alpha_{k}<\alpha^{*}$, it holds $\alpha_{k-1}\ge \alpha^{*}$ by Proposition~\ref{prop} (no two consecutive iterates are both strictly below $\alpha^{*}$). 
According to the update rules, such $\alpha_k$ is generated by \eqref{eq:alpha sequence} and satisfies either \textbf{Case 1.2} or \textbf{Case 2.2}. By Lemma~\ref{Lemma:case1} and \ref{Lemma:case2}, we have
\begin{equation}
    \label{eq:1.22.2}
\alpha^{*}-\alpha_{k}\le \max\{M,M'\}(\alpha_{k-1}-\alpha^{*})
\end{equation}
  Given that $\alpha_{k}\le \alpha^{*}\le\alpha_{k-1}\le \alpha_{0}$, together with \eqref{eq:1.22.2}, we derive
  \begin{equation}
    \label{eq: negative bound}\alpha_{-1}\ge\alpha^{*}\ge \alpha_{k}\ge \alpha^{*}-\max\{M,M'\}(\alpha_{0}-\alpha^{*})=\alpha_{\text{low}}.
  \end{equation}
\end{itemize}
Combining \eqref{eq: positive bound} and \eqref{eq: negative bound}, we conclude that \[
\alpha_{-1}\ge \alpha_{k}\ge \alpha_{\text{low}}, \,k=0,1,\cdots,
\]
which completes the proof.
\end{proof}
\begin{assumption}
\label{assumption 1}
Let $g$  be sufficiently differentiable and nonlinear on $[\alpha_{\text{low}},\alpha_{-1}]$.
\end{assumption}
The regularity assumption is 
 imposed to derive an accurate convergence rate of \eqref{eq:alpha sequence}.
It is straightforward to satisfy the nonlinearity assumption.
If \(g\) degenerates to a linear function over \([\alpha^{*}, \alpha_{-1}]\). Our algorithm yields
\[
\alpha_{0} = \alpha_{-1} - \frac{g(\alpha_{-1})}{f_{2}(x_{-1})}=\alpha^{*}.
\]  
achieving one-step convergence. 
Here we define a key parametric
\begin{equation}
\label{eq:tau}
\tau = \max_{\omega \in [\alpha^{*}, \alpha_{-1}]} |g''(\omega)|,
\end{equation}
which is positive under Assumption~\ref{assumption 1}.
Via \eqref{eq:form Dinkelbach}, an explicit convergence rate for $\{\alpha_{p}^{+}\}$
 can be directly deduced as below.

\begin{corollary}
\label{corollary}
Under Assumption~\ref{assumption 1}, for any $p\ge 2$, it holds that 
\[
\alpha_{p}^{+}-\alpha^{*}\le \frac{\tau}{f_{2}(x^{*})}(\alpha_{p-1}^{+}-\alpha^{*})^{2},
\]
where $\tau$ is defined in \eqref{eq:tau}.
\end{corollary}
\begin{proof}
    For any $p\ge 2$, we derive
    \begin{align}
    \alpha_{p}^{+}-\alpha^{*}&\le \alpha_{p-1}^{+}-\frac{g(\alpha_{p-1}^{+})}{f_{2}(x_{p-1}^{+})}-\alpha^{*}\ (\text{by}\ \eqref{eq:form Dinkelbach})\nonumber\\
    &=(\alpha_{p-1}^{+}-\alpha^{*})\left(1-\frac{\delta g(\alpha_{p-1}^{+},\alpha^{*})}{f_{2}(x_{p-1}^{+})}\right)\nonumber\\
    &\le(\alpha_{p-1}^{+}-\alpha^{*})\frac{f_{2}(x_{p-1}^{+})-f_{2}(x^{*})}{f_{2}(x_{p-1}^{+})} \ (\text{since}\ \delta g(\alpha_{p-1}^{+},\alpha^{*})\ge f_{2}(x^{*}))\nonumber\\
    &\le(\alpha_{p-1}^{+}-\alpha^{*})\frac{g'(\alpha_{p-1}^{+})-g'(\alpha^{*})}{f_{2}(x^{*})}\ (\text{since}\ f_{2}(x_{p-1}^{+})\ge f_{2}(x^{*}))\nonumber\\
    &=(\alpha_{p-1}^{+}-\alpha^{*})^{2}\frac{g''(\omega_{p-1})}{f_{2}(x^{*})}\ (\text{for some} \ \omega_{p-1}\in [\alpha^{*},\alpha_{p-1}^{+}])\nonumber
    \\
    &\le \frac{\tau}{f_{2}(x^{*})}(\alpha_{p-1}^{+}-\alpha^{*})^{2} \ (\text{by}\ \eqref{eq:tau}).\nonumber
    \end{align}
    \end{proof}
\subsection{Asymptotic Convergence Rate Analysis}
This subsection is devoted to establishing the precise convergence rates of accelerated Dinkelbach method  Algorithm~\ref{Algorithm: Accelerated Dinkelbach}.
All subsequent analysis is conducted under Assumption~\ref{assumption 1}.
Our theoretical framework develops as follows.
\begin{itemize}
    \item 
Lemma~\ref{lemma:mechanism} establishes a  sufficient condition under which \eqref{eq:cond} holds.
\item Lemmas~\ref{lemma:forgotten} and~\ref{lemma:1} characterize the convergence rates of iterative schemes, namely \eqref{eq:min tangent} and \eqref{eq:alpha sequence}, respectively.
\item Corollary~\ref{coro_1} determines the position of $\alpha_{k+1}$ generated by \eqref{eq:alpha sequence}.
\item 
Theorem~\ref{Theorem_final} synthesizes the aforementioned  results to establish the precise asymptotic convergence rates of the algorithm.
\end{itemize}

\begin{lemma}
\label{lemma:mechanism}
    For any predefined $\rho>1$, there exists a $k_{\rho}$ such that for all $k\ge k_{\rho}$, if $\alpha^{*}<\alpha_{k}<\alpha_{k-1}$, then \eqref{eq:cond} holds.
\end{lemma}
\begin{proof}
    Since $\{\alpha_{p}^{+}\}$ converges to $\alpha^{*}$, there exists a $p_{\rho}$ and a corresponding
$k_{\rho}$ such that 
\begin{equation}
\label{eq:2002}
0\le \alpha_{k_{\rho}}-\alpha^{*}=\alpha_{p_{\rho}}^{+}-\alpha^{*}\leq \frac{f_{2}(x^{*})^{2}}{\rho\tau f_{2}(x_{-1})}, 
\end{equation}
where $\tau$ is defined in \eqref{eq:tau}. The fraction 
 is well-defined since $\rho>1$, $\tau>0$ (guaranteed by Assumption~\ref{assumption 1}) and $f_{2}>0$.
For any $k > k_{\rho}$ with $\alpha_{k-1} > \alpha_{k} > \alpha^{*}$, there exists an index $p'$ satisfying $\alpha_{p'}^{+}=\alpha_{k}$ and $ \alpha_{p'-1}^{+}=\alpha_{k-1}$. 
We establish the following chain of inequalities:
\begin{align}
\frac{\alpha_{k} - \alpha^{*}}{\alpha_{k-1} - \alpha^{*}}=\frac{\alpha_{p'}^+ - \alpha^{*}}{\alpha_{p'-1}^+ - \alpha^{*}} &\leq \frac{\tau}{f_{2}(x^{*})}(\alpha_{p'-1}^+ - \alpha^{*}) \ (\text{by Corollary~\ref{corollary}})  \nonumber\\&\le \frac{\tau}{f_{2}(x^{*})}(\alpha_{k_{\rho}} - \alpha^{*}) \  (\text{from the monotonicity of}\ \{\alpha_{p}^{+}\} )  \nonumber\\&\le \frac{\tau}{f_{2}(x^{*})}\frac{f_{2}(x^{*})^{2}}{\rho\tau f_{2}(x_{-1})}\ (\text{by}\ \eqref{eq:2002} )  \nonumber\\
&=\frac{f_{2}(x^{*})}{\rho f_{2}(x_{-1})}\nonumber\\
&\le \frac{f_{2}(x^{*})}{\rho f_{2}(x_{k-1})}\frac{f_{2}(x_{k})}{f_{2}(x_{k})}\ (\text{since} \ f_{2}(x_{-1})\ge f_{2}(x_{k-1})) \nonumber\\
&\le \frac{\delta g(\alpha_{k-1},\alpha^{*})}{\rho\delta g(\alpha_{k},\alpha_{k-1})}\frac{f_{2}(x_{k})}{\delta g(\alpha_{k},\alpha^{*})}.\label{eq::1}
\end{align}
The last inequality follows from Lemma~\ref{lemma:subgradient_property}, which gives 
\[
 f_2(x^*)\le \delta g(\alpha_{k-1}, \alpha^*), \quad f_2(x_{k-1})\ge \delta g(\alpha_{k-1}, \alpha_k), \quad
  f_2(x_k)\ge \delta g(\alpha_k, \alpha^*)
.
\]
Since $\delta g>0$  by Corollary~\ref{coro_supply},  $\rho>0$ and $\alpha^{*}<\alpha_{k}<\alpha_{k-1}$, \eqref{eq::1} is equivalent to
\begin{equation}
    \label{eq__1}
\rho(\alpha_{k}-\alpha^{*})\delta g(\alpha_{k},\alpha_{k-1})\delta g(\alpha_{k},\alpha^{*})\le \delta g(\alpha_{k-1}-\alpha^{*})f_{2}(x_{k})(\alpha_{k-1}-\alpha^{*}).
\end{equation}
 Multiplying both sides of \eqref{eq__1} by  $\alpha_{k}-\alpha_{k-1}$ ($\le 0$),  results in 
\[
\rho(g(\alpha_{k})-g(\alpha_{k-1}))(g(\alpha_{k})-g(\alpha^{*}))\ge (g(\alpha_{k-1})-g(\alpha^{*}))f_{2}(x_{k})(\alpha_{k}-\alpha_{k-1})
\]
By $g(\alpha^*) = 0$, the above is further equivalent to \eqref{eq:cond}, 
which completes the proof.
\end{proof}
We now examine the convergence rates of the iteration scheme~\eqref{eq:min tangent} and~\eqref{eq:alpha sequence} and study the propagation behavior of subsequent iterates.
\begin{lemma} 
\label{lemma:forgotten}
The iteration scheme \eqref{eq:min tangent} achieves a convergence rate of at least 
2. Specifically, for 
$k$ such that 
$\alpha_{k+2}$ is generated via \eqref{eq:min tangent}, the inequality
\begin{equation}
|\alpha_{k+2}-\alpha^{*}|\le C_{1}|\alpha_{k+1}-\alpha^{*}|^{2}
\end{equation} 
holds, where $C_{1}$ denotes a positive constant.
\end{lemma}

\begin{proof}
    By Proposition~\ref{prop}, 
generating of $\alpha_{k+2}$ via \eqref{eq:min tangent} implies $\alpha_{k+1} \leq \alpha^{*}$ and $\alpha_{k+2} \geq \alpha^{*}$. We derive the following chain of inequalities:
\begin{align}
\alpha_{k+2}-\alpha^{*} &= \min\left\{\alpha_{k+1}-\frac{g(\alpha_{k+1})}{f_{2}(x_{k+1})}, \alpha_{k}-\frac{g(\alpha_{k})}{f_{2}(x_{k})}\right\} -\alpha^{*}\nonumber \\
&\leq \alpha_{k+1}-\frac{g(\alpha_{k+1})}{f_{2}(x_{k+1})} {\color{black}-\alpha^{*}}\nonumber \\
&= (\alpha^{*}-\alpha_{k+1})\left(\frac{\delta g(\alpha_{k+1},\alpha^{*})-f_{2}(x_{k+1})}{f_{2}(x_{k+1})}\right) \nonumber \\
&\leq \frac{\max_{\xi \in [\alpha_{k+1},\alpha^{*}]}|g''(\xi)|}{\min_{x \in \mathcal{F}}f_{2}(x)}(\alpha_{k+1}-\alpha^{*})^{2}\label{eq:63}\\
&\le \frac{\max_{\xi \in [\alpha_{\text{low}},\alpha^{-1}]}|g''(\xi)|}{\min_{x \in \mathcal{F}}f_{2}(x)}(\alpha_{k+1}-\alpha^{*})^{2}, \nonumber\\
&=C_{1}(\alpha_{k+1}-\alpha^{*})^{2}\ ( C_{1} \triangleq \frac{\max_{\xi \in [\alpha_{\text{low}},\alpha^{-1}]}|g''(\xi)|}{\min_{x \in \mathcal{F}}f_{2}(x)}).\label{eq_C_1}
\end{align}
Here, the inequality \eqref{eq:63} follows from  $\delta g(\alpha_{k+1},\alpha^{*}) \leq f_{2}(x^{*})$ and the twice continuous differentiability of $g(\alpha)$. 
Finally, \eqref{eq_C_1} concludes the proof.
\end{proof}
\begin{lemma}\label{lemma:1}
    {\color{black}
    Consider the characterization parameter defined by}
\begin{equation}\label{eq:chi}
\chi = 3g''^{2}(\alpha^{*}) - 2g'(\alpha^{*})g'''(\alpha^{*}).
\end{equation}
{\color{black}
For sufficiently large $k$, 
the convergence behavior of iteration~\eqref{eq:alpha sequence} depends on \(\chi\) in the following way.}
\begin{itemize}
    \item If $\chi \neq 0$, then  \eqref{eq:alpha sequence} achieves super-quadratic convergence:
    \begin{equation}
    |\alpha_{k+1} - \alpha^{*}| \le C_2 |\alpha_{k} - \alpha^{*}|^{2} |\alpha_{k-1} - \alpha^{*}|,
    \end{equation}
    where $C_2 > 0$ is a constant independent of $k$.
    
    \item If $\chi = 0$, then  \eqref{eq:alpha sequence} attains a higher-order super-quadratic convergence:
    \begin{equation}
    |\alpha_{k+1} - \alpha^{*}| \le C_3 |\alpha_{k} - \alpha^{*}|^{2} |\alpha_{k-1} - \alpha^{*}|^{2},
    \end{equation}
    where $C_3 > 0$ is a constant independent of $k$.
\end{itemize}
\end{lemma}

\begin{proof}
Following the methodology in \cite{FernandezTorres2015}, we expand \eqref{eq:alpha sequence} via third-order Taylor series around the root $\alpha^{*}$, yielding
\begin{equation}
\label{eq_1+}
\alpha_{k+1} - \alpha^{*}= (\alpha_{k} - \alpha^{*})^{2} (\alpha_{k-1} - \alpha^{*}) \frac{g''(\alpha^{*}) - \frac{2}{3}g'''(\xi_{k})g'(\alpha^{*}) + O(|\alpha_{k-1}-\alpha^{*}|)}{4g'^{2}(\alpha^{*}) + O(|\alpha_{k-1}-\alpha^{*}|)},
\end{equation}
{\color{black}
where $\xi_{k}$ lies between $\alpha^{*}$ and $\alpha_{k-1}$,  as defined by the Taylor expansion:}
\begin{align}
g(\alpha_{k-1}) &= g'(\alpha^{*})(\alpha_{k-1}-\alpha^{*}) + \frac{g''(\alpha^{*})}{2}(\alpha_{k-1}-\alpha^{*})^{2} + \frac{g'''(\xi_{k})}{6}(\alpha_{k-1}-\alpha^{*})^{3}.\nonumber
\end{align} 
Rewrite \eqref{eq_1+} as
\begin{equation}
\label{eq_Pk}
\alpha_{k+1}-\alpha^{*}=P_{k}(\alpha_{k}-\alpha^{*})^{2}(\alpha_{k-1}-\alpha^{*}),
\end{equation}
where \( P_k \) is an expression depending on \( \alpha_k \) and \( \alpha_{k-1} \), i.e., it takes the form given in: 
\begin{align}
    P_{k}&= \frac{g''(\alpha^{*}) - \frac{2}{3}g'''(\xi_{k})g'(\alpha^{*}) + O(|\alpha_{k-1}-\alpha^{*}|)}{4g'^{2}(\alpha^{*}) + O(|\alpha_{k-1}-\alpha^{*}|)}\nonumber\\
&=\frac{3g''(\alpha^{*})-2g'''(\alpha^{*})g'(\alpha^{*})+ 2(g'''(\alpha^{*})-g'''(\xi_{k}))g'(\alpha^{*}) + O(|\alpha_{k-1}-\alpha^{*}|)}{12g'^{2}(\alpha^{*}) + O(|\alpha_{k-1}-\alpha^{*}|)}\nonumber\\
    &=\frac{\chi+ 2(g'''(\alpha^{*})-g'''(\xi_{k}))g'(\alpha^{*}) + O(|\alpha_{k-1}-\alpha^{*}|)}{12g'^{2}(\alpha^{*}) + O(|\alpha_{k-1}-\alpha^{*}|)}.\label{eq_P}
\end{align}
By Assumption~\ref{assumption 1}, we derive the upper estimate for $|g'''(\alpha^{*})-g'''(\xi_{k})|$:
\begin{equation}
\label{eq_1}
    |g'''(\alpha^{*})-g'''(\xi_{k})|\le \max_{\zeta \in[\alpha_{\text{low}},\alpha_{-1}]}g^{(4)}(\zeta)|\alpha_{k-1}-\alpha^{*}|=O(|\alpha_{k-1}-\alpha^{*}|).
\end{equation}
Thus, 
If $\chi\neq0$,  we obtain
\begin{align}
|P_{k}|&=\left|\frac{\chi+ 2(g'''(\alpha^{*})-g'''(\xi_{k}))g'(\alpha^{*}) + O(|\alpha_{k-1}-\alpha^{*}|)}{12g'^{2}(\alpha^{*}) + O(|\alpha_{k-1}-\alpha^{*}|)}\right|\nonumber\\
&\le \frac{|\chi|+ |O(|\alpha_{k-1}-\alpha^{*}|)|}{12g'^{2}(\alpha^{*}) - |O(|\alpha_{k-1}-\alpha^{*}|)|}\, (\text{by}\ \eqref{eq_1})\label{eq__-}\\
&\le\frac{|\chi|}{3g'^{2}(\alpha^{*})}\label{eq_taking2}.
\end{align}
We now prove the last inequality in \eqref{eq_taking2}. 
By the definition of big-O notation, there exists a constant $d>0$, such that 
\begin{equation}
    \label{eq__O}
|O(|\alpha_{k-1}-\alpha^{*}|)|\le d|\alpha_{k-1}-\alpha^{*}|.
\end{equation}
Since \(\alpha_k \to \alpha^*\), \(|\alpha_{k-1} - \alpha^*| \to 0\) as \(k \to \infty\), 
so for  sufficiently large $k$,
\begin{equation}\label{eq__O1}|\alpha_{k-1}-\alpha^{*}|<\frac{\min\{|\chi|,6g'^{2}(\alpha^{*})\}}{d}.
\end{equation}
Substituting \eqref{eq__O1} into \eqref{eq__O} gives
\[
|O(|\alpha_{k-1}-\alpha^{*}|)|<\min\{|\chi|,6g'^{2}(\alpha^{*})\}. 
\]
Thus, the numerator in \eqref{eq__-} satisfies $|\chi|+ |O(|\alpha_{k-1}-\alpha^{*}|)|<|\chi|+|\chi|=2|\chi|$, and the denominator satisfies $12g'^{2}(\alpha^{*}) - |O(|\alpha_{k-1}-\alpha^{*}|)|>12g'^{2}(\alpha^{*})-6g'^{2}(\alpha^{*})=6g'^{2}(\alpha^{*})$.
 Combining these bounds, we have 
\[ \frac{|\chi| + |O(|\alpha_{k-1} - \alpha^*|)|}{12g'^2(\alpha^*) - |O(|\alpha_{k-1} - \alpha^*|)|} < \frac{2|\chi|}{6g'^2(\alpha^*)} = \frac{|\chi|}{3g'^2(\alpha^*)},  \] which confirms that \eqref{eq_taking2} holds.
Substituting \eqref{eq_taking2} into \eqref{eq_Pk} derives
\begin{equation}
    \label{eq_chi_0_P_k}
    |\alpha_{k+1}-\alpha^{*}|\le \frac{|\chi|}{3g'^{2}(\alpha^{*})}|\alpha_{k}-\alpha^{*}|^{2}|\alpha_{k-1}-\alpha^{*}|.
\end{equation}
 
When \(\chi = 0\), we perform a higher-order Taylor expansion of \( g(\alpha_k) \) and \( g(\alpha_{k-1}) \) around \(\alpha^*\) in the numerator of \( P_k \) (which depends on \(\alpha_{k-1}\) and \(\alpha_k\)). 
{\color{black} This yields:
}
\[
P_k = \frac{\sum_{m=0}^N \sum_{n=0}^N \mu_{m,n} \left( \alpha_{k-1} - \alpha^* \right)^m \left( \alpha_k - \alpha^* \right)^n + o\left( \left| \alpha_{k-1} - \alpha^* \right|^N \left| \alpha_k - \alpha^* \right|^N \right)}{12 g'^2(\alpha^*) + O\left( \left| \alpha_{k-1} - \alpha^* \right| \right)}, \tag{79}
\]
where \(N \in \mathbb{N}^+\) is the expansion order, 
 \(\mu_{m,n} \in \mathbb{R}\) are coefficients determined by the derivatives of the numerator's underlying function at \((\alpha^*,\alpha^*)\) (independent of $k$), and 
\(o\left( \left| \alpha_{k-1} - \alpha^* \right|^N \left| \alpha_k - \alpha^* \right|^N \right)\) denotes the higher-order remainder term 
satisfying \(\frac{o\left( \left| \alpha_{k-1} - \alpha^* \right|^N \left| \alpha_k - \alpha^* \right|^N \right)}{\left| \alpha_{k-1} - \alpha^* \right|^N \left| \alpha_k - \alpha^* \right|^N} \to 0\) as \((\alpha_{k-1}, \alpha_k) \to (\alpha^*, \alpha^*)\).
From \eqref{eq_P}, and the condition $\chi=0$, we have $\mu_{0,0}=0$. 
Let \((m^*, n^*)\) be the lexicographically smallest pair in \([0, N] \times [0, N]\) such that \(\mu_{m,n} \neq 0\) and \(m + 2n\) is minimized.
 If no such pair exists for any sufficiently large \(N\), then \(P_k = 0\) and consequently \(\alpha_{k+1} = \alpha^*\), which trivially satisfies the lemma.

We now establish an upper bound for $P_{k}$:
\begin{align}
    |P_{k}|&=\left|\frac{ \sum_{m=0}^{N}\sum_{n=0}^{N} \mu_{m,n} \left( \alpha_{k-1} - \alpha^* \right)^m \left( \alpha_k - \alpha^* \right)^n+o\bigl(|\alpha_{k-1}-\alpha^{*}|^{N}|\alpha_{k}-\alpha^{*}|^{N}\bigr)}{12 \, g'^2(\alpha^*) + O\bigl( |\alpha_{k-1} - \alpha^*| \bigr)}\right|\nonumber\\
    &\le \frac{\sum_{m=0}^{N}\sum_{n=0}^{N} |\mu_{m,n}| \left| \alpha_{k-1} - \alpha^* \right|^m \left| \alpha_k - \alpha^* \right|^n}{6 \, g'^2(\alpha^*)}\label{eq_sum2}\\
    &\le \frac{\sum_{m=0}^{N}\sum_{n=0}^{N} |\mu_{m,n}| \left| \alpha_{k-1} - \alpha^* \right|^m \text{max}\{C_{1},\frac{\tau}{f_{2}(x^{*})}\}\left| \alpha_{k-1} - \alpha^* \right|^{2n}}{6 \, g'^2(\alpha^*)}\label{eq_sum3}\\ 
&=\text{max}\{C_{1},\frac{\tau}{f_{2}(x^{*})}\}\frac{\sum_{m=0}^{N}\sum_{n=0}^{N} |\mu_{m,n}| \left| \alpha_{k-1} - \alpha^* \right|^{m+2n}}{6 \, g'^2(\alpha^*)}\label{eq_=}\\
&\le \frac{\text{max}\{C_{1},\frac{\tau}{f_{2}(x^{*})}\}|\mu_{m^{*},n^{*}}| \left| \alpha_{k-1} - \alpha^* \right|^{m^{*}+2n^{*}}}{3g'(\alpha^{*})}\label{eq_mn*}\\
&\le \frac{\text{max}\{C_{1},\frac{\tau}{f_{2}(x^{*})}\}|\mu_{m^{*},n^{*}}| \left| \alpha_{k-1} - \alpha^* \right|}{3g'(\alpha^{*})}\ (\text{using } m^{*}+2n^{*}\ge 1),\label{eq_nosum}
\end{align}
where \eqref{eq_sum2} is derived by analogous reasoning to that in \eqref{eq_taking2},  and \eqref{eq_sum3} uses Corollary~\ref{corollary} and Lemma~\ref{lemma:forgotten}  under the two cases ($\alpha^{*} < \alpha_k < \alpha_{k-1}$ or $\alpha_{k-1} < \alpha^{*} < \alpha_k$).

{\color{black}To establish inequality \eqref{eq_mn*}, we bound each term in the double summation in \eqref{eq_=} and leveraging finite summation properties. 
{\color{black}Let  \( \mu_{\text{max}} = \max\left\{ |\mu_{m,n}| \mid 0 \leq m,n \leq N \right\}>0 \) (as $\mu_{m^{*},n^{*}}>0$) be a  finite constant.} For sufficiently large \( k \), we have \begin{equation}
    \label{eq_minmax} |\alpha_{k-1} - \alpha^*| < \min\left\{1,\frac{|\mu_{m^{*},n^{*}}|}{((N+1)^{2}-1)\mu_{\text{max}}} \right\}, \end{equation} 
    Then for any \( (m,n)\neq (m^{*},n^{*})\), each term in the summation thus satisfies
\begin{align}
|\mu_{m,n}| |\alpha_{k-1} - \alpha^*|^{m+2n} &\leq \mu_{\text{max}} |\alpha_{k-1} - \alpha^*|^{m+2n}\ (\text{by the definition of } \mu_{\text{max}})\nonumber\\
&\leq \mu_{\text{max}} |\alpha_{k-1} - \alpha^*|^{m^{*}+2n^{*}+1}\ (\text{by the definition of }m^{*},\,n^{*})\nonumber\\
& =\mu_{\text{max}}|\alpha_{k-1}-\alpha^{*}||\alpha_{k-1}-\alpha^{*}|^{m^{*}+2n^{*}}\nonumber\\
&\le\frac{|\mu_{m^{*},n^{*}}||\alpha_{k-1}-\alpha^{*}|^{m^{*}+2n^{*}}}{(N+1)^{2}-1}\ (\text{using \eqref{eq_minmax}}).\label{eq_bound_1}
\end{align}Substituting \eqref{eq_bound_1} into the double summation  \eqref{eq_=} 
for each \( (m,n)\in  [0, N] \times [0, N] \setminus \{(m^*, n^*)\}\), which contains \( (N+1)^2 - 1 \) terms, we obtain:}
\begin{align}
|P_k| &\leq \text{max}\{C_{1},\frac{\tau}{f_{2}(x^{*})}\}\frac{|\mu_{m^{*},n^{*}}||\alpha_{k-1}-\alpha^{*}|^{m^{*}+2n^{*}}+ |\mu_{m^{*},n^{*}}|\left| \alpha_{k-1} - \alpha^* \right|^{m^{*}+2n^{*}}}{6g'^2(\alpha^*)}\nonumber\\
&=\frac{\max\left\{ C_1, \frac{\tau}{f_2(x^*)} \right\} |\mu_{m^*, n^*}| |\alpha_{k-1} - \alpha^*|^{m^* + 2n^*}}{3g'^2(\alpha^*)},
\end{align}
confirming \eqref{eq_mn*}.

Substituting \eqref{eq_nosum} into \eqref{eq_Pk} derives
\begin{equation}
    \label{eq_chi_1_P_k}
|\alpha_{k+1}-\alpha^{*}|\le \frac{\text{max}\{C_{1},\frac{\tau}{f_{2}(x^{*})}\}|\mu_{m^{*},n^{*}}|}{3g'^{2}(\alpha^{*})}|\alpha_{k}-\alpha^{*}|^{2}|\alpha_{k-1}-\alpha^{*}|^{2}.
\end{equation}
Finally, taking $C_{2}=\chi/6g''^{2}(\alpha^{*})$ and $C_{3}=\text{max}\{C_{1},\frac{\tau}{f_{2}(x^{*})}\}|\mu_{m^{*},n^{*}}|/3g'^{2}(\alpha^{*})$, we complete the proof.

\end{proof}

\begin{corollary}
    \label{coro_1}
    If $\chi\neq 0$, the position of $\alpha_{k+1}$ generated by \eqref{eq:alpha sequence} is determined by the type of \textbf{Case 1} and \textbf{Case 2}.
 \begin{itemize}
    \item If $\chi > 0$:
    \begin{itemize}
        \item When $\alpha_k$ and $\alpha_{k-1}$ satisfy \textbf{Case 1}, it leads to \textbf{Case 1.2} ($\alpha_{k+1} < \alpha^{*}$).
        \item When $\alpha_k$ and $\alpha_{k-1}$ satisfy \textbf{Case 2}, it leads to \textbf{Case 2.1} ($\alpha_{k+1} \geq \alpha^{*}$).
    \end{itemize}
    
    \item If $\chi < 0$:
    \begin{itemize}
        \item When $\alpha_k$ and $\alpha_{k-1}$ satisfy \textbf{Case 1}, it leads to \textbf{Case 1.1} ($\alpha_{k+1} \geq \alpha^{*}$).
        \item When $\alpha_k$ and $\alpha_{k-1}$ satisfy \textbf{Case 2}, it leads to \textbf{Case 2.2} ($\alpha_{k+1} < \alpha^{*}$).
    \end{itemize}
\end{itemize}
\end{corollary}
\begin{proof}
  By \eqref{eq_Pk}, we have 
  \begin{equation}
    \label{eq_9}
  \text{sgn}(\alpha_{k+1}-\alpha^{*})=\text{sgn}(P_{k}(\alpha_{k}-\alpha^{*})^{2}(\alpha_{k-1}-\alpha^{*}))=\text{sgn}(P_{k})\text{sgn}(\alpha_{k-1}-\alpha^{*})\left|\text{sgn}(\alpha_{k}-\alpha^{*})\right|,
  \end{equation}
  where 
  \[
\operatorname{sgn}(x) = 
\begin{cases} 
1 & \text{if } x > 0, \\
0 & \text{if } x = 0, \\
-1 & \text{if } x < 0.
\end{cases}
\]
  Since $\chi\neq 0$, it follows from \eqref{eq_P} and \eqref{eq_1} that 
 \begin{equation}
    \label{eq_10}
  \text{sgn}(P_{k})=\text{sgn}(\chi)
 \end{equation}
  holds for sufficient large $k$.
  Substituting \eqref{eq_10} into \eqref{eq_9} gives
\begin{equation}
    \label{eq_11}
\text{sgn}(\alpha_{k+1}-\alpha^{*})=\text{sgn}(\chi)\text{sgn}(\alpha_{k-1}-\alpha^{*})\left|\text{sgn}(\alpha_{k}-\alpha^{*})\right|.
\end{equation}
The corollary follows directly from \eqref{eq_11}, which completes the proof.
\end{proof}
\begin{theorem}
    \label{Theorem_final}
For any $x_{-1}\in\mathcal{F}$, let \(\{\alpha_k\}\) be the sequence generated by the accelerated Dinkelbach algorithm in Algorithm~\ref{Algorithm: Accelerated Dinkelbach}.  
When \(\chi \neq 0\) (defined in \eqref{eq:chi}), \(\{\alpha_k\}\) exhibits asymptotic periodicity 
with respect to  its convergence rates.  Let \(L\) denote the period of this asymptotic periodicity. The average convergence order per iteration is as follows:
\begin{itemize}
    \item If $\chi>0$, $L=1 \ \text{or}\ 2$ and the average convergence order per iteration is $1+\sqrt{2}\ \text{or}\  \sqrt{5}$, respectively.
    \item If $\chi<0$, $L=3$ and the average convergence order per iteration is \(\sqrt[3]{12}\).
\end{itemize}
When \(\chi = 0\), if \(\{\alpha_{q}^{-}\}\) is a finite set, the asymptotic average convergence order of \(\{\alpha_k\}\) is at least \(1+\sqrt{3}\). Conversely, if \(\{\alpha_{q}^{-}\}\) is an infinite set, the asymptotic average convergence order of \(\{\alpha_k\}\) is at least \( \sqrt{6} \).
\end{theorem}
    \begin{proof}
Assume $\{\alpha_k\}$ is an infinite sequence. 
First, we consider the case when $\chi \neq 0$. 
 For a sufficiently large $k_{0}$ (Lemma~\ref{lemma:mechanism}, Lemma~\ref{lemma:1}, and Corollary~\ref{coro_1} are valid for all $k\ge k_{0}$), by the positions of $\alpha_{k_{0}-1}$ and $\alpha_{k_{0}}$, we  determine 
the subsequent generations.

The analysis is partitioned into two $\chi$-sign cases: 
    (1) $\chi >0$, (2) $\chi < 0$
and three positional configurations: 
(i) $\alpha^* < \alpha_{k_{0}} < \alpha_{k_{0}-1}$,
(ii) $\alpha_{k_{0}} < \alpha^* < \alpha_{k_{0}-1}$,
(iii) $\alpha_{k_{0}-1} < \alpha^* < \alpha_{k_{0}}$.
We summarize in Table \ref{table:Convergence rates} the length of period under the partitioned configurations described above, followed by rigorous derivation of these results.
\begin{table}[ht]
\centering
\setlength{\tabcolsep}{3pt} 
\caption{Length of period $L$  based on signal of $\chi$ and $\alpha^{*}$, $\alpha_{k_{0}-1}$, and $\alpha_{k_{0}}$ positions }
\label{table:Convergence rates}
    \begin{tabular}{cccc} 
\toprule

& (i) $\alpha^* < \alpha_{k_{0}} < \alpha_{k_{0}-1}$ & (ii) $\alpha_{k_{0}} < \alpha^* < \alpha_{k_{0}-1}$ & (iii) $\alpha_{k_{0}-1} < \alpha^* < \alpha_{k_{0}}$ \\
\midrule
$(1)\,\chi > 0$ & $1$ &$2$&  $2$  \\
\hline
$(2)\,\chi < 0$ & $3$ & $3$ & $3$ \\
\bottomrule
\end{tabular}
\end{table}
\begin{itemize}
    \item[1-i)]($\chi> 0$ and $\alpha^{*}<\alpha_{k_{0}}<\alpha_{k_{0}-1}$):
    By our update rules,
    $\alpha_{k_{0}+1}$ is generated by \eqref{eq:alpha sequence}. Corollary~\ref{coro_1} guarantees $\alpha_{k_{0}+1} \ge \alpha^{*}$ (\textbf{Case 2.1}). Lemma~\ref{lemma:mechanism} ensures that $\alpha_{k_{0}+1}$ and $\alpha_{k_{0}}$ satisfy \eqref{eq:cond}, so $\alpha_{k_{0}+2}$ is generated by
    \eqref{eq:alpha sequence} and also falls into
    \textbf{Case 2.1} by Corollary~\ref{coro_1}. Thus, we confirm this subcase exhibits a period $\{\alpha_{k_{0}+1+n}\}_{n=0,1,\cdots }$, where each $\alpha_{k_{0}+1+n}$ is generated by \eqref{eq:alpha sequence}.
    \item[1-ii)]($\chi> 0$ and $\alpha_{k_{0}}<\alpha^{*}<\alpha_{k_{0}-1}$): By our update rules,  $\alpha_{k_{0}+1}$ is generated via \eqref{eq:min tangent} and $\alpha_{k_{0}+1}>\alpha^{*}$. Subsequently, $\alpha_{k_{0}+2}$ is generated by  \eqref{eq:alpha sequence} and  falls into \textbf{Case 1.2} ($\alpha_{k_{0}+2}< \alpha^{*}$) by Lemma~\ref{coro_1}.  Then, $\alpha_{k_{0}+3}\ge \alpha^{*}$ via \eqref{eq:min tangent} and $\alpha_{k_{0}+4}$ (constructed by~\eqref{eq:alpha sequence}) also falls into \textbf{Case 1.2}. This subcase exhibits a period $\{\alpha_{k_{0}+1+2n},\alpha_{k_{0}+2+2n}\}_{n=0,1,\cdots }$, where each $\alpha_{k_{0}+1+n}$ and $\alpha_{k_{0}+2+n}$ are generated by \eqref{eq:min tangent} and \eqref{eq:alpha sequence}, respectively.
    \item[1-iii)]($\chi> 0$ and $\alpha_{k_{0}-1}<\alpha^{*}<\alpha_{k_{0}}$): This subcase  coincides with 1-ii) after reindexing by letting $k_{0}=k_{0}+1$. 
     \item[2-i)]($\chi< 0$ and $\alpha^{*}<\alpha_{k_{0}}<\alpha_{k_{0}-1}$): By our update rules, 
     $\alpha_{k_{0}+1}$ is generated by \eqref{eq:alpha sequence}. Corollary~\ref{coro_1} implies $\alpha_{k_{0}+1}< \alpha^{*}$ (\textbf{Case 2.2}). Subsequently,
     $\alpha_{k_{0}+2}$ follows \eqref{eq:min tangent}. $\alpha_{k_{0}+3}$ is generated by \eqref{eq:alpha sequence} (\textbf{Case 1.1}) and $\alpha_{k_{0}+3}\ge \alpha^{*}$. Then $\alpha_{k_{0}+4}\le \alpha^{*}$ is also given by \eqref{eq:alpha sequence} (\textbf{Case 2.2}). This subcase gives a 3-step period $\{\alpha_{k_{0}+2+3n},\alpha_{k_{0}+3+3n}, \alpha_{k_{0}+4+3n}\}_{n=0,1,\cdots}$, where each $\alpha_{k_{0}+2+3n},\alpha_{k_{0}+3+3n}, \,\text{and}\,\alpha_{k_{0}+4+3n}$ are generated via \eqref{eq:min tangent}, \eqref{eq:alpha sequence} and \eqref{eq:alpha sequence}, respectively.
    \item[2-ii)]($\chi< 0$ and $\alpha_{k_{0}}<\alpha^{*}<\alpha_{k_{0}-1}$): This subcase  coincides with 2-i) after reindexing by letting $k_{0}=k_{0}+1$.  
    \item[2-iii)]($\chi< 0$ and $\alpha_{k_{0}-1}<\alpha^{*}<\alpha_{k_{0}}$): This subcase  coincides with 2-i) after reindexing by letting $k_{0}=k_{0}+2$.  
\end{itemize}
Now we derive the average convergence order as below.
\begin{itemize}
    \item If $L=1$ (1-i), then by Theorem~\ref{lemma:1+sqrt2}, each step achieves a convergence order of $1+\sqrt{2}$.
    \item If $L=2$ (1-ii, 1-iii), by above, for $n=0,1,\cdots $, we have
    \begin{align}
    |\alpha_{k_{0}+1+2n}-\alpha^{*}|&\le C_{1}|\alpha_{k_{0}+2n}-\alpha^{*}|^{2} \ (\text{by Lemma~\ref{lemma:forgotten}}),\label{eq_31} \\
    |\alpha_{k_{0}+2+2n}-\alpha^{*}|&\le C_{2}|\alpha_{k_{0}+1+2n}-\alpha^{*}|^{2}|\alpha_{k_{0}+2n}-\alpha^{*}| \ (\text{by Lemma~\ref{lemma:1}}).\label{eq_32}
    \end{align}
    Substituting \eqref{eq_31} into \eqref{eq_32} yields
    \[
    |\alpha_{k_{0}+2+2n}-\alpha^{*}|\le C_{1}^{2}C_{2}|\alpha_{k_{0}+2n}-\alpha^{*}|^{5}, \ n=0,1,\cdots,\]
    which implies that each step achieves an average convergence order of $\sqrt{5}$.
    \item If $L=3$ (2-i, 2-ii, 2-iii), for  $n=0,1,\cdots $, we have 
    \begin{align}
        |\alpha_{k_{0}+2+3n}-\alpha^{*}|&\le C_{1}|\alpha_{k_{0}+1+3n}-\alpha^{*}|^{2} \ (\text{by Lemma~\ref{lemma:forgotten}}),\label{eq_33} \\
    |\alpha_{k_{0}+3+3n}-\alpha^{*}|&\le C_{2}|\alpha_{k_{0}+2+3n}-\alpha^{*}|^{2}|\alpha_{k_{0}+1+3n}-\alpha^{*}| \ (\text{by Lemma~\ref{lemma:1}}),\label{eq_34}\\
    |\alpha_{k_{0}+4+3n}-\alpha^{*}|&\le C_{2}|\alpha_{k_{0}+3+3n}-\alpha^{*}|^{2}|\alpha_{k_{0}+2+3n}-\alpha^{*}| \ (\text{by Lemma~\ref{lemma:1}}).\label{eq_35}
    \end{align}
    Substituting \eqref{eq_33} and \eqref{eq_34} into \eqref{eq_35} gives
    \[
    |\alpha_{k_{0}+4+3n}-\alpha^{*}|\le C_{1}^{5}C_{2}^{3}|\alpha_{k_{0}+1+3n}|^{12},\ n=0,1,\cdots,\]
    which derives that each step achieves an average order of $\sqrt[3]{12}$.
\end{itemize}

Second, we consider $\chi=0$. 
If $\{\alpha_{q}^{-}\}$ (corresponding to negative $g$-values) is finite, then there exists a sufficiently large $k_{0}$ such that $\alpha_{k}\ge \alpha^{*}$ for all $k\ge k_{0}$. By Lemma~\ref{lemma:1}, we obtain for $n=0,1,\cdots$, 
\begin{equation}
\label{eq_1+sqrt3}
    |\alpha_{k_{0}+n}-\alpha^{*}|\le C_{3}|\alpha_{k_{0}+n-1}-\alpha^{*}|^{2}|\alpha_{k_{0}+n-2}-\alpha^{*}|^{2}.
\end{equation}
Let $\beta$ denote 
the asymptotic order. 
 This recurrence relation \eqref{eq_1+sqrt3} implies that \(\beta\) satisfies the characteristic equation \(\beta^2 - 2\beta - 2 = 0\), whose larger root is \(1 + \sqrt{3}\). Consequently, the average convergence order is at least \(1 + \sqrt{3} > \sqrt{5}\).

Subsequently, suppose that $\{\alpha_{q}^{+}\}$ is infinite. For a sufficiently large $k_0$, let $\{\alpha_{k_q}\}$ denote the sequence $\{\alpha_{k_0+q}^{-}\}$ with $q=0,1,\cdots$. By Proposition~\ref{prop}, we have $k_{q+1} - k_q \ge 2$. 
{\color{black}Applying Lemmas \ref{lemma:forgotten} and \ref{lemma:1}, 
we derive the following chain of inequalities for the iterations from $k_q$ to $k_{q+1}$:}
\begin{align}
|\alpha_{k_q+1} - \alpha^*| &\le C_1 |\alpha_{k_q} - \alpha^*|^2, \nonumber\\
|\alpha_{k_q+2} - \alpha^*| &\le C_3 |\alpha_{k_q+1} - \alpha^*|^2 |\alpha_{k_q} - \alpha^*|^2, \nonumber\\
&\vdots \nonumber\\
|\alpha_{k_{q+1}} - \alpha^*| &\le C_3 |\alpha_{k_{q+1}-1} - \alpha^*|^2 |\alpha_{k_{q+1}-2} - \alpha^*|^2. \nonumber
\end{align}
By recursively substituting the first $k_{q+1} - k_q - 1$ inequalities into the last one, we obtain
\[
|\alpha_{k_{q+1}} - \alpha^*| \leq C_1^{d_q} \cdot C_3^{e_q} \cdot \left| \alpha_{k_q} - \alpha^* \right|^{f_q},
\]
where the exponents satisfy
\begin{align}
d_q &= \frac{(1 + \sqrt{3})(1 + \sqrt{3})^{k_{q+1}-k_q-1} + (\sqrt{3}-1)(1 - \sqrt{3})^{k_{q+1}-k_q-1}}{2\sqrt{3}}, \nonumber\\
e_q &= \frac{(\sqrt{3} + 2)(1 + \sqrt{3})^{k_{q+1}-k_q-1} + (\sqrt{3} - 2)(1 - \sqrt{3})^{k_{q+1}-k_q-1}-1}{3}, \nonumber\\
f_q &= \frac{(\sqrt{3} + 2)(1 + \sqrt{3})^{k_{q+1}-k_q-1} + (\sqrt{3} - 2)(1 - \sqrt{3})^{k_{q+1}-k_q-1}}{\sqrt{3}}. \nonumber
\end{align}
It is straightforward to verify that 
\[
d_{q},e_{q}\le f_{q}, \, \forall k_{q+1}-k_q\ge 2.\]
Let
\[
\overline{C_{1}}=\max\{1,C_{1}\},\  \overline{C_{3}}=\max\{1,C_{3}\}.\]
We have the following upper estimate of $|\alpha_{k_{q+1}}-\alpha^{*}|$: 
\begin{align}
|\alpha_{k_{q+1}}-\alpha^{*}|&\le 
 C_1^{d_q} \cdot C_3^{e_q} \cdot \left| \alpha_{k_q} - \alpha^* \right|^{f_q}\nonumber\\
 &\le 
\overline{C_{1}}^{d_q} \cdot \overline{C_{3}}^{e_q} \cdot \left| \alpha_{k_q} - \alpha^* \right|^{f_q}\,(\text{using }\overline{C_{1}}\ge C_{1},\,\overline{C_{3}}\ge C_{3})\nonumber\\
 &\le \overline{C_{1}}^{f_q} \cdot \overline{C_{3}}^{f_q} \cdot \left| \alpha_{k_q} - \alpha^* \right|^{f_q}\,(\text{using }\overline{C_{1}}\ge 1,\,\overline{C_{3}}\ge 1\text{, and }d_{q},e_{q}\le f_{q})\nonumber
 \\&=\left(\overline{C_{1}}\overline{C_{3}}( \alpha_{k_q} - \alpha^* )\right)^{f_q},\nonumber
\end{align}
which indicates that the average convergence order is $(f_{q})^{\frac{1}{k_{q+1}-k_{q}}}$.
{\color{black}
Now we determine a lower bound for this average order.

Rewriting $f_q$ as
\[
f_q = \frac{\sqrt{6}^{k_{q+1}-k_q}}{\sqrt{3}} \left( \frac{\sqrt{3}+2}{1+\sqrt{3}} \left( \frac{1+\sqrt{3}}{\sqrt{6}} \right)^{k_{q+1}-k_q} + \frac{\sqrt{3}-2}{\sqrt{3}-1} \left( \frac{1-\sqrt{3}}{\sqrt{6}} \right)^{k_{q+1}-k_q} \right).
\]
We observe that $(1+\sqrt{3})/\sqrt{6} > 1$, $|1-\sqrt{3}|/\sqrt{6} < 1$, and the coefficient of the second term is negative. For $k_{q+1} - k_q \geq 2$, we have
\begin{equation}
\nonumber
f_q \geq \frac{\sqrt{6}^{k_{q+1}-k_q}}{\sqrt{3}} \left( \frac{\sqrt{3}+2}{1+\sqrt{3}} \left( \frac{1+\sqrt{3}}{\sqrt{6}} \right)^2 + \frac{\sqrt{3}-2}{\sqrt{3}-1} \left( \frac{1-\sqrt{3}}{\sqrt{6}} \right)^2 \right) = \sqrt{6}^{k_{q+1}-k_q}. 
\end{equation}
 Therefore, the average convergence order per iteration satisfies:
\[
f_q^{1/(k_{q+1}-k_q)} \geq \sqrt{6}.
\]
This completes the proof.
}
\end{proof}

Subsequently, we give some numerical examples. The calculations are performed in Mathematica 12.3.1.0.

\begin{example}
Table~\ref{tab:rho_convergence} illustrates the effects of varying $\rho$ values on the convergence process for the function $g(\alpha) = e^{0.5\alpha} + 5\alpha - 9$ with initial value $\alpha_{-1} = 11$. 
{\color{black}
The results indicate that the algorithm is not very sensitive to the specific value of $\rho$, although values close to 1 generally yield better performance.}

Moreover, 
the result also reveals an asymptotic, periodic acceleration pattern.
Given that $g'(\alpha)=\frac{1}{2}e^{0.5\alpha}+5$, $g''(\alpha)=\frac{1}{4}e^{0.5\alpha}$, and $g'''(\alpha)=\frac{1}{8}e^{0.5\alpha}$,  we can compute
\begin{align}
\chi&=3g''^{2}(\alpha^{*})-2g'(\alpha^{*})g'''(\alpha^{*})\nonumber\\
&=\frac{1}{16}e^{\frac{1}{2}\alpha^{*}}(e^{\frac{1}{2}\alpha^{*}}-20)< 0\nonumber.
\end{align}
{\color{black}
The observed period of 3 matches the theoretical prediction for the case \(\chi<0\) in Theorem \ref{Theorem_final}.}
\end{example}

    \begin{table}[htbp]
\centering
\caption{Convergence behavior  with various $\rho$ values}
\label{tab:rho_convergence}
\sisetup{
  exponent-product = \times,
  table-format = -1.3e-1,
  table-number-alignment = center,
  table-space-text-pre = \textminus,
  table-align-exponent = false
}
\begin{tabular}{c 
    S[table-format=1.4e1] 
    S[table-format=1.4e1] 
    S[table-format=-1.4e1] 
    S[table-format=-1.4e1] 
    S[table-format=-1.4e1]
}
\toprule
\multirow{2}{*}{$k$} & \multicolumn{5}{c}{$\rho$} \\
\cmidrule(lr){2-6}
 & {1.00001} & {1.1} & {2} & {10} & {100} \\
\midrule

-1  & 2.91e2  & 2.91e2  & 2.91e2  & 2.91e2  & 2.91e2  \\
0  & 1.13e2  & 1.13e2  & 1.13e2  & 1.13e2  & 1.13e2  \\
1  & 1.61e1  & 1.61e1  & 4.36e1  & 4.36e1  & 4.36e1  \\
2  & -1.17e0 & -1.17e0 & 1.36e1  & 1.36e1  & 1.36e1  \\
3  & 9.22e-3 & 9.22e-3 & 1.62e0  & 1.62e0  & 1.62e0  \\
4  & 1.43e-8 & 1.43e-8 & -2.47e-1 & 1.90e-2 & 1.90e-2 \\
5  & -2.74e-22 & -2.74e-22 & 4.21e-4 & -2.27e-4 & 2.51e-6 \\
6  & 5.25e-46 & 5.25e-46 & 6.33e-12 & 3.60e-10 & -3.33e-10 \\
7  & 1.09e-116 & 1.09e-116 & -2.45e-30 & 4.26e-27 & 7.71e-22 \\
8  & -9.08e-282 & -9.08e-282 & 4.17e-62 & -9.47e-67 & 2.87e-56 \\
9 & \multicolumn{1}{c}{—}    & \multicolumn{1}{c}{—}   & 6.15e-157 & 6.25e-135 & -9.19e-137 \\
10 &    \multicolumn{1}{c}{—}      &   \multicolumn{1}{c}{—}       & -2.78e-339 &    5.35e-339      & 5.89e-275 \\
\bottomrule
\end{tabular}
\end{table}

\begin{example}
Table~\ref{tab:g_values_shifted_init} illustrates different period lengths determined by distinct initial points. For simplicity, we define $g$ as 
\begin{align}
g(\alpha)&=\alpha \arctan(\alpha) + 2\alpha - \frac{1}{2}\ln(1+\alpha^2),\nonumber
\end{align}
The initial points are set to  $5$ and $7$ and  $\rho$ is set to 1.00001.
Clearly, $g(0)=0$. Since
$g'(\alpha)=\arctan(\alpha)+2$, $g''(\alpha)=1/(\alpha^{2}+1)$ and $g'''(\alpha)=-2\alpha/(\alpha^{2}+1)^2$, $g$ is convex and increases on $[0,\infty)$.  We compute $\chi$ (defined in \eqref{eq:chi}) as follows.
\[
\chi=3g''^{2}(0)-2g'(0)g'''(0)=3>0.
\]
Notice that when $\alpha_{-1}=-5$, $g(\alpha_{4})\le 0$ (i.e., $\alpha_{4}\le \alpha^{*}$), which coincides with $L=2$ in Theorem~\ref{Theorem_final}. However, when $\alpha_{-1}=7$, our algorithm exhibits a period 
$L=1$, which coincides with the case specified in Theorem \ref{Theorem_final}.
\end{example}
\begin{table}[h]
\centering 
\caption{Values of \( g(\alpha_k) \) with Different Initial Values}
\label{tab:g_values_shifted_init}
\begin{tabular}{@{}c>{\centering\arraybackslash}c>{\centering\arraybackslash}c@{}} 
\toprule
\( k \) & \( \alpha_{-1} = 5 \): $g(\alpha_{k})$ & \( \alpha_{-1} = 7 \): $g(\alpha_{k})$ \\
\midrule
-1   & \( 1.52 \times 10^{1} \)              & \( 2.20 \times 10^{1} \) \\
0   & \( 1.08\times 10^{0}  \)              & \( 1.30 \times 10^{0} \) \\
1   & \( 6.70\times 10^{-2}  \)            & \( 9.45 \times 10^{-2} \) \\
2   & \( 7.62\times 10^{-5}  \)         & \( 1.80 \times 10^{-4} \) \\
3   & \( 6.11 \times 10^{-12} \) & \( 4.82 \times 10^{-11} \) \\
4   & \( -2.19 \times 10^{-28} \) & \( 9.10 \times 10^{-28} \) \\
\bottomrule
\end{tabular}
\end{table}
\begin{example}
This example compares our algorithm and original Dinkelbach method under special case ($\chi=0$). Specifically, let $g$ be 
    \[
    g(\alpha)=\alpha^{3}+\alpha^{2}+\alpha.
    \]
    Clearly $g(0)=0$. 
    With $g'(\alpha)=3\alpha^{2}+2\alpha+1$, $g''(\alpha)=6\alpha+2$ and $g'''(\alpha)=6$
    one can readily check that $g$ is convex and increasing on $[-1/3,10]$. Furthermore,
    \[
   \chi=3g''^{2}(0)-2g'(0)g'''(0)=0.
    \]
    We set the initial point to $10$ and $\rho=1.00001$. Table~\ref{tab:iteration_results} summarizes the results. For $k=9,10$, the convergence order is $90/33\approx 2.727$,  $246/90\approx 2.733$, which is approximately $1+\sqrt{3}$. This example is consistent with Theorem~\ref{Theorem_final} for the subcase $\chi=0$.
\end{example}
\begin{table}[h]
\centering 
\caption{Iteration results for Accelerated Dinkelbach method and original Dinkelbach method }\label{tab:iteration_results}
\begin{tabular}{@{}c>{\centering\arraybackslash}c>{\centering\arraybackslash}c@{}} 
\toprule
$k$ & $g(\alpha_{k})$ for Accelerated Dinkelbach & $g(\alpha_{k})$ for Dinkelbach \\
\midrule
-1 & $1.11 \times 10^{3}$ & $1.11 \times 10^{3}$ \\
0 & $3.29 \times 10^{2}$ & $3.29 \times 10^{2}$ \\
1 & $5.75 \times 10^{1}$ & $9.78 \times 10^{1}$ \\
2 & $1.12 \times 10^{1}$ & $2.91 \times 10^{1}$ \\
3 & $2.12 \times 10^{0}$ & $8.71 \times 10^{0}$ \\
4 & $3.60 \times 10^{-1}$ & $2.60 \times 10^{0}$ \\
5 & $2.56 \times 10^{-2}$ & $7.53 \times 10^{-1}$ \\
6 & $3.93 \times 10^{-5}$ & $1.84 \times 10^{-1}$ \\
7 & $9.36 \times 10^{-13}$ & $2.35 \times 10^{-2}$ \\
8 & $1.35 \times 10^{-33}$ & $5.26 \times 10^{-4}$ \\
9 & $1.60 \times 10^{-90}$ & $2.76 \times 10^{-7}$ \\
10 & $4.69 \times 10^{-246}$ & $7.64 \times 10^{-14}$ \\
11 & — & $5.84 \times 10^{-27}$ \\
12 & — & $3.41 \times 10^{-53}$ \\
13 & — & $1.16 \times 10^{-105}$ \\
14 & — & $1.35 \times 10^{-210}$ \\
\bottomrule
\end{tabular}
\end{table}

\section{Conclusion}\label{section:Conclusion}
This work significantly advances both theoretical foundations and computational efficiency of Dinkelbach-type methods for fractional programming. Building upon the classical Dinkelbach method (1967) and its interval variant (1991), we develop two accelerated frameworks.

{\color{black}
The accelerated interval Dinkelbach method produces monotonic and convergent sequences of upper and lower bounds, establishing superquadratic and cubic convergence rates, respectively, under the assumption that the parametric function $g$ is twice continuously differentiable.}

The accelerated Dinkelbach method, by contrast, generates a non-monotonic sequence that converges globally to the optimal value.
Under the assumption that $g$ is sufficiently differentiable, 
the average convergence order per iterate attains at least
 $\sqrt{5}$. 
 Moreover, for almost all $g$, the sequences generated by accelerated Dinkelbach method 
exhibit  asymptotic periodicity on the convergence order. 

These advancements overcome longstanding limitations in convergence rates while preserving computational tractability.
Important directions for future research include both theoretical refinements (simplifying the accelerated Dinkelbach algorithm and its convergence proof) and empirical validation (conducting more large-scale experiments to verify algorithmic stability).

Finally, we propose two open problems: 
\begin{enumerate}
\item Is it possible to achieve a \textit{higher order} of convergence without introducing substantial additional computational overhead?
\item Does there exist a monotonically accelerated Dinkelbach algorithm with higher convergence order?
\end{enumerate}

\section{Declarations}
\begin{itemize}
    
    \item Fund: This research was supported by Beijing Natural Science Foundation (QY24112) and the National Natural Science Foundation of China (12171021).
\end{itemize} 
\bibliography{main}%

\end{document}